\documentclass[12pt, reqno]{amsart}
\usepackage{amsmath, amsthm, amscd, amsfonts, amssymb, graphicx, color, mathrsfs}
\usepackage[bookmarksnumbered, colorlinks, plainpages]{hyperref}
\usepackage[all]{xy}
\usepackage{slashed}

\usepackage{color}
\usepackage{fancybox}
\usepackage{colordvi}
\usepackage{multicol}
\usepackage{colordvi}

\newtheorem{theorem}{Theorem}[section]
\newtheorem{lemma}[theorem]{Lemma}
\newtheorem{proposition}[theorem]{Proposition}
\newtheorem{corollary}[theorem]{Corollary}
\theoremstyle{definition}

\newtheorem{example}[theorem]{Example}

\theoremstyle{remark}
\newtheorem{remark}[theorem]{Remark}
\numberwithin{equation}{section}

\begin{document}
\setcounter{page}{1}

\title[ The Dixmier trace, and the residue on compact manifolds]{The Dixmier trace and the noncommutative residue   for multipliers on compact manifolds}

\author[D. Cardona]{Duv\'an Cardona}
\address{
  Duv\'an Cardona:
  \endgraf
  Department of Mathematics  
  \endgraf
  Pontificia Universidad Javeriana
  \endgraf
  Bogot\'a
  \endgraf
  Colombia
  \endgraf
  {\it E-mail address} {\rm d.cardona@uniandes.edu.co;
duvanc306@gmail.com}
  }

\author[C. del Corral]{C\'esar del Corral}
\address{
  C\'esar del Corral:
  \endgraf
  Department of Mathematics  
  \endgraf
  Universidad de los Andes
  \endgraf
  Bogot\'a
  \endgraf
  Colombia
  \endgraf
  {\it E-mail address} {\rm ce-del@uniandes.edu.co; cesar.math@gmail.com}
  }


\dedicatory{}

\keywords{Dixmier trace; non commutative residue; global operators; representation theory; closed manifolds, manifolds with boundary; H\"ormander classes, Boutet de Monvel's algebra}

\date{Received: ;  Revised: ; Accepted: .
}

\begin{abstract}
In this paper we give formulae for the Dixmier trace and the noncommutative residue (also called Wodzicki's residue) of  pseudo-differential operators by using the notion of global symbol. We consider both cases,  compact manifolds with or without boundary. Our analysis on  the  Dixmier trace of invariant pseudo-differential operators on closed manifolds will be based on the Fourier analysis associated to every elliptic and positive operator and the quantization process developed by Delgado and Ruzhansky. In particular, for compact Lie groups this can be done by using the representation theory of the group in view of the Peter-Weyl theorem and the Ruzhansky-Turunen symbolic calculus.  The analysis of invariant pseudo-differential operators on compact manifolds with boundary  will be based on the global calculus of pseudo-differential operators developed by Ruzhansky and Tokmagambetov.\\
\textbf{MSC 2010.} Primary { 58J40; Secondary 58B34, 47L20, 22E30}.
\end{abstract} \maketitle

\tableofcontents
\section{Introduction}

The Dixmier trace arises in functional analysis, from the problem if there exists a non-trivial trace -- which does not coincide with the spectral trace -- for an ideal containing the set of trace class operators  on a Hilbert space $H$. In 1966, J. Dixmier in \cite{Dixmier} shows that there exists a trace which vanishes in the ideal of  trace class operators, but it is non-trivial in the so called Dixmier ideal $\mathcal{L}^{1,\infty}(H).$  
As it was pointed out in  \cite{SU} and references therein, the Dixmier trace is important due to its applications in fractal theory,  foliation theory, spaces of non-commuting coordinates, perturbation theory, non-commutative geometry and
quantum field theory.\\

 In this paper we  study the Dixmier trace of invariant  pseudo-differential operators on compact manifolds (multipliers) with or without boundary.  By applying the Connes trace theorem (see \cite{Connes}) we derive a formula for the noncommutative residue of Fourier multipliers on compact Lie groups, and manifolds with boundary.  This is possible based in the recent quantizations of pseudo-differential operators associated to global  symbols  introduced  by M. Ruzhansky and V. Turunen \cite{Ruz} (for compact Lie groups),  J. Delgado and M. Ruzhansky \cite{DR1} (for compact manifolds without boundary), and finally,  by  M. Ruzhansky and N. Tokmagambetov \cite{DRTk} (for compact manifolds with boundary). One of the advantages of our approach is that we obtain results on the Dixmier trace and the noncommutative residue for Fourier multipliers in terms  of global (full) symbols, where we give a new viewpoint, instead of the  classical results where the problem was treated by using local symbols (see e.g.  \cite{Connes,Dixmier,Fedosov,GGIS,Schrohe2,Schrohe,Wodzicki} and references therein). We summarize our research in the following results:
\begin{itemize}
\item We provide necessary and sufficient conditions in order to guarantee that global invariant pseudo-differential operators  on a manifold  $M$ with or without boundary lie in the {\em Dixmier class} $\mathcal{L}^{(1,\infty)}(H),$ $H=L^2(M),$ (see equation \eqref{dixmier}).  
\item  We provide  sufficient and necessary conditions in order to obtain Dixmier traceability for a type of  pseudo-differential operators with global symbols in H\"ormander classes on compact Lie groups and we express our results by using the representation theory of these groups. Also, we use Connes' trace theorem in order to show formulae for the noncommutative residue of classical pseudo-differential operators on compact Lie groups.
\item On a compact Lie group, we provide  formulae in terms of global symbols for the noncommutative residue of a type of classical pseudo-differential operators (non necessarily Fourier multipliers) in terms of the representation theory of the group (see Proposition \ref{noninvarianttrace}).
\item For a compact manifold $M$ with or without boundary,  we find criteria for global symbols in order that the corresponding operators belong to the Marcinkiewicz ideal $\mathcal{L}^{(p,\infty)}(H),$ $1<p<\infty,$ where $H=L^2(M),$ (see equation \eqref{lpinftydef}).
\end{itemize}

In order to present our main results we precise some definitions. By following Connes \cite{Connes}, if $H$ is a Hilbert space, the class $\mathcal{L}^{(1,\infty)}(H)$ consists of those bounded linear operators $A\in\mathcal{L}(H)$ satisfying
\begin{equation}
\sum_{1\leq n\leq N}s_{n}(A)=O(\log(N)),\,\,\,N\rightarrow \infty,
\end{equation}
where $\{s_{n}(A)\}$ denotes the sequence of singular values of $A$,  i.e. the square roots of the eigenvalues of the non-negative self-adjoint operator $A^\ast A.$ 
So, $\mathcal{L}^{(1,\infty)}(H)$ is endowed with the norm
\begin{equation}\label{dixmier}
\Vert A \Vert_{\mathcal{L}^{(1,\infty)}(H)}=\sup_{N\geq 2}\frac{1}{\log(N)}\sum_{1\leq n\leq N}s_{n}(A).
\end{equation}
The Dixmier trace $\textnormal{Tr}_{\omega}$ on positive operators in $\mathcal{L}^{(1,\infty)}(H)$ can be formulated by using  an average function $\omega$ of the following form
$$\lim _{\omega }\frac{1}{\log(N)}\sum_{1\leq n\leq N}s_{n}(A):=\omega\big( \frac{1}{\log(N)}\sum_{1\leq n\leq N}s_{n}(A)\big),$$
where $\lim_{\omega}$ is  a positive linear functional on $l^{\infty}(\mathbb{N})$, the set of bounded sequences, satisfying: 
\begin{itemize}
\item $\lim_{\omega}(\alpha_n)\geq 0$ if all $\alpha_n\geq 0$;
\item $\lim_{\omega}(\alpha_n) = \lim(\alpha_n)$ whenever the ordinary limit exists;
\item $\lim_{\omega}(\alpha_1,\alpha_1, \alpha_2, \alpha_2,\alpha_3,\alpha_3,\dots) = \lim_{w}(\alpha_n)$.
\end{itemize}
The functional $\textnormal{Tr}_{\omega}$ can be defined for positive operators, and later it can be extended to the whole ideal $\mathcal{L}^{(1,\infty)}(H)$ by linearity (in this case $\textnormal{Tr}_\omega$ is not necessarily a positive functional). The functional $\textnormal{Tr}_{\omega}$  is trivial  on the ideal $\mathcal{L}^1(H)$ of trace class operators, (see \cite{Connes} or  \cite{Dixmier}). The subset of $\mathcal{L}^{(1,\infty)}(H)$ with Dixmier trace independent of the average function $\omega$ defines the class of Dixmier measurable operators.\\
\\
In  noncommutative geometry, a remarkable result due to A. Connes shows that  classical pseudo-differential operators acting on $L^2(M)$ with order $-\dim(M),$ belong to the Dixmier  class, and the Dixmier trace coincides with the noncommutative residue  (see \cite{Wodzicki}), which is the unique trace (up to by factors) in the algebra of pseudo-differential operators with classical symbols on a closed manifold, (see \cite[pag. 305]{Connes} for this fact). R. Nest, E. Schrohe in \cite{Nest-Schrohe}, show that in general Connes' result does not hold for the Boutet de Monvel algebra on a manifold with boundary (cf. (\ref{Eq:Noncomm.res.boundary}) and (\ref{Eq.Dixmier-trace-Bdy})).\\
\\
More generally, the ideal $\mathcal{L}^{(p,\infty)}(H)$ consists of those linear bounded operators $A$ on $H$ satisfying the condition:
\begin{equation}\label{lpinftydef}
\sum_{1\leq n\leq N}s_{n}(A)=O(N^{(1-1/p)}),\,\,N\rightarrow \infty,
\end{equation}
for  $1<p<\infty.$ On $\mathcal{L}^{(p,\infty)}(H)$ the usual norm is given by   
\begin{equation}
\Vert A \Vert_{\mathcal{L}^{(p,\infty)}(H)}:=\sup_{N\geq 1}  N^{(\frac{1}{p})-1}\sum_{1\leq n\leq N}s_{n}(A).
\end{equation}
In order to present our results we introduce briefly the theory of pseudo-differential operators that we use here \cite{DR,Hor2,Ruz}. The notion of global symbol on $\mathbb{R}^n$ is natural because a pseudo-differential operator $A$ is an integral operator  defined by
\begin{equation}
Af(x)=\int_{\mathbb{R}^n} e^{ix\xi}\sigma_A(x,\xi)\hat{f}(\xi)\, d\xi,
\end{equation}
for a suitable smooth function $\sigma_A(x,\xi)$ -- called the symbol of $A$ -- and satisfying some bounded conditions on its derivatives (see \cite{Hor2}). {Here $\hat{f}$ denotes the euclidean Fourier transform of the function $f$.   If we consider a closed manifold  $M$ (i.e a compact manifold without boundary) and $H=L^{2}(M),$   there exists a (global) Fourier analysis associated to every positive elliptic pseudo-differential operator $E$ on $M$ which gives for certain  pseudo-differential operators $A$ --called {\em $E$-invariants}-- a discrete {Fourier} representation of the form
\begin{equation}\label{DRqua}
Af(x)=\sum_{l=0}^{\infty}\langle\sigma_{A,E}(l)\widehat{f}(l),e_{l}(x)\rangle_{\mathbb{C}^{d_l}}, 
\end{equation}
where $e_{l}(x):=(e_{l}^{m})_{1\leq m\leq d_{l}}$ and $\{e^{m}_{l}:l\in \mathbb{N}, 1\leq m\leq d_{l}\}$ provides a basis of $L^{2}(M)$ consisting  of eigenfunctions $e^{m}_{l},$ associated to certain eigenvalues $\lambda_l,$ $l\in\mathbb{N}_0,$ of the operator $E.$ The function $\sigma_{A,E}$ is  called the \textit{matrix valued symbol} of $A$ with respect to $E$. The process associating  to every operator $A$ acting in $C^{\infty}(M)$ a such matrix valued symbol $\sigma_{A,E}$ was introduced by J. Delgado and M. Ruzhansky, \cite{DR5}.  Also,  If $M=G$ is a compact Lie group and $E=-\mathcal{L}_{G}$ is minus the  Laplace-Beltrami operator on $G,$ Ruzhansky and Turunen \cite{Ruz} give a discrete representation to every pseudo-differential operator $A$  on $C^{\infty}(G)$ in terms of the representation theory of the group $G$, in the following way 
\begin{equation}\label{globalsymbol}
Af(x)=\sum_{[\xi]\in\widehat{G}}d_{\xi}\textnormal{Tr}[\xi(x)\sigma_A(x,\xi)\widehat{f}(\xi)],
\end{equation}
here $\widehat{G}$ denotes the unitary dual of the group $G.$  Ruzhansky-Turunen's calculus gives a characterization  of the H\"ormander classes $\Psi^{m}_{\rho,\delta}(G)$ on a compact Lie group $G$ by using the notion of global symbols. In fact, $A\in \Psi^{m}_{\rho,\delta}(G)$ if and only if its global symbol $\sigma_{A}(x,\xi)$ as in  \eqref{globalsymbol} satisfies
\begin{equation}
\Vert \Delta_{\xi}^{\alpha}\partial_{x}^{\beta}\sigma_{A}(x,\xi) \Vert_{op}\leq C_{\alpha,\beta}\langle \xi\rangle^{m-\rho|\alpha|+\delta|\beta|},\,\,\,x\in G, [\xi]\in\widehat{G},
\end{equation}
where $\langle \xi\rangle:=(1+\lambda_{[\xi]})^{\frac{1}{2}}$ and $\{\lambda_{[\xi]}:[\xi]\in \widehat{G}\}$ is the spectrum  of $-\mathcal{L}_{G}.$ In Remark \ref{mainremarkglobal} we will recall the relationship between the symbols $\sigma_{A,-\mathcal{L}(G)}$ and $\sigma_{A}$ associated to the quantizations \eqref{DRqua} and \eqref{globalsymbol}, respectively.\\
\\
Now, we recall that  for a  compact manifold $M$ (without boundary) of dimension $\varkappa$, a classical pseudo-differential operator $A$ on $M$ of order $m$, can be defined by local symbols. This means that for any local chart $U$, the operator $A$ has the form
$$Au(x)=\int_{T^{*}_xU} e^{ix\xi}\sigma^A(x,\xi)\widehat{u}(\xi)\, d\xi$$
where  $\sigma^A(x,\xi)$  is a smooth function on $T^{*} U\cong U\times\mathbb{R}^n,$} $T^{*} _{x}U=\mathbb{R}^n$,  admitting an asymptotic expansion 
\begin{equation}\label{Eq:asym-homg}
\sigma^A(x,\xi)\sim \sum_{j=0}^{\infty}\sigma^{A_{m-j}}(x,\xi)
\end{equation}
where $\sigma_{m-j}(x,\xi)$ are homogeneous functions in $\xi\neq 0,$ of degree $m-j$ for $\xi$ far  from zero. The set of classical pseudo-differential operators of order $m$ is denoted by $\Psi^m_{cl}(M)$. If $A\in\Psi_{cl}(M)$, for all $x\in M$, $\int_{|\xi|=1}\sigma_{-\varkappa}(x,\xi)\, d\xi\,dx$ defines a local density which can be glued over $M$. So, the non-commutative residue, is the functional defined on classical operators by 
\begin{equation}\label{resi}
\textnormal{res}\,(A)=\frac{1}{\varkappa(2\pi)^\varkappa}\int_{M}\int_{| \xi|=1}\sigma^{A_{-\varkappa}}(x,\xi)\, d\xi\,dx.
\end{equation}
An important feature of the  non-commutative residue is that it vanishes on non-integer order classical operators.  A complementary trace to the noncommutative residue is the canonical trace,  $A\mapsto \textnormal{TR}(A),$ on the set of classical pseudo-differential. The term complementary is justified because the noncommutative residue is defined on the whole algebra of pseudo-differential operators, but does not extend the usual $L^2$-trace. However the canonical trace is not well defined on the whole algebra of classical pseudo-differential operators, but it extends the usual $L^2$-trace on a suitable set, see, e.g. \cite{Paycha,Paycha-Scott,GSch04,CardonaA}.\\

Now we present our main results. Our formulae for Dixmier traces will be presented for positive operators in $\mathcal{L}^{(1,\infty)}(L^2(M))$ because such formulae can be extended to the whole ideal  $\mathcal{L}^{(1,\infty)}(L^2(M))$  by linearity.  We start with the case of manifolds without boundary. Here, $\varkappa$ denotes the dimension of a compact manifold without boundary $M.$

\begin{theorem}\label{main1}
Let $M$ be a  $\varkappa$-dimensional compact manifold without boundary   and let $E\in \Psi^{\nu}_{+e}(M)$ be a positive elliptic pseudo-differential operator on $M.$ If  $A: L^2(M)\rightarrow L^2(M)$  is a $E$-invariant bounded operator  with matrix-valued symbol $(\sigma_{A,E}(l))_{l}$,  then we have, 
\begin{itemize}
\item $A$ is \textrm{Dixmier traceable}, i.e. $A\in \mathcal{L}^{(1,\infty)}(L^{2}(M))$ if and only if 
\begin{equation}\label{eq1main}
\tau(A):=\frac{1}{\dim(M)}\lim_{N\rightarrow\infty} \frac{1}{\log N} \sum_{l: (1+\lambda_{l})^{\frac{1}{\nu}}\leq N} \textnormal{Tr}(|\sigma_{A,E}(l)|)<\infty,
\end{equation}
where $\sigma_{A,E}$ are as in \eqref{DRqua}. Moreover, if $A$ is positive, $\tau(A)=\textnormal{Tr}_{w}(A).$
\item The $E$-invariant operator $A\in \mathcal{L}^{(p,\infty)}(L^2(M))$ if and only if
\begin{equation}
\gamma_{p}(A):= \sup_{N\geq 1}\,\, N^{\dim M(\frac{1}{p}-1)}\cdot\sum_{l: (1+\lambda_{l})^{\frac{1}{\nu}}\leq N}\textnormal{Tr}(|\sigma_{A,E}(l)|)<\infty,
\end{equation}
for all $1<p<\infty.$ In this case, the norms $\Vert A\Vert_{\mathcal{L}^{(p,\infty)}(L^{2}(M))}$ and $\gamma_{p}(A)$ are equivalents, which we denote by $\Vert A\Vert_{\mathcal{L}^{(p,\infty)}(L^{2}(M))}\asymp \gamma_{p}(A)$.
\item Let $M=G$ be a compact Lie group and let $\widehat{G}$ be the unitary dual of $G.$ If we denote by $\sigma_{A}(x,\xi)$ the matrix valued symbol associated to $A,$ then under the condition
\begin{equation}
\Vert \Delta_{\xi}^{\alpha}\partial_{x}^{\beta}\sigma_{A}(x,\xi) \Vert_{op}\leq C\langle \xi\rangle^{-\varkappa-|\alpha|},\,\,\,x\in G, [\xi]\in\widehat{G},\,\,\, \forall \alpha\in \mathbb{R}^n,
\end{equation}
the operator $A$ is Dixmier measurable. Moreover, if $A$ is left-invariant, then $A$ is Dixmier traceable if and only if
\begin{equation}
\tau(A):=\frac{1}{\dim(G)}\lim_{N\rightarrow\infty}\frac{1}{\log N}\sum_{[\xi]:\langle \xi\rangle\leq N}d_{\xi}\textnormal{Tr}(|\sigma_{A}(\xi)|)<\infty.
\end{equation}
In this case, if $A$ is positive, $\tau(A)=\textnormal{Tr}_{w}(A).$
\item If $A\in \Psi^{-\varkappa}_{cl}(G)$ is a classical, positive and left-invariant pseudo-differential operator, the noncommutative residue of $A$ (cf. \eqref{resi}), is given in terms of representations on $G$ by
\begin{equation}\label{resi0}
\textnormal{res}(A)=\frac{1}{\dim(G)}\lim_{N\rightarrow\infty}\frac{1}{\log N}\sum_{[\xi]:\langle \xi\rangle\leq N}d_{\xi}\textnormal{Tr}(|\sigma_{A}(\xi)|).
\end{equation} 
 Moreover, if $A$ is as above, but not necessarily a multiplier, and its symbol admits an asymptotic expansion in homogeneous components of the form:
\begin{equation}
\sigma^A(x,\xi)\sim \sum_{j=0}^{\infty}a_{m-j}(x)\sigma^{A_{m-j}}(\xi),
\end{equation}
then 
\begin{equation}\label{cardona-delcorral}
\textnormal{res}(A)=\frac{1}{\dim(G)}\int_{G}a_{-\varkappa}(x)dx\times\lim_{N\rightarrow\infty}\frac{ 1  }{\log N}\sum_{\xi: \langle \xi\rangle\leq N}d_{\xi}\textnormal{Tr}(|\sigma_{A_{-\varkappa}}(\xi)|).
\end{equation} 
\end{itemize}
\end{theorem}

It is important to mention that  our formula \eqref{eq1main} for the Dixmier trace of a global operator $A$ on a closed manifold $M$  has two components:   one is, the inverse of the dimension of $M,$ which has  a geometric nature. The other one has an analytic nature, which is defined by the global symbol $\sigma_{A,E}(\cdot)$ of $A$, determined by the Fourier analysis associated to $E$.\\
We observe that the approach in the preceding result can be used in order to analyze the Dixmier traceability of Fourier multipliers on compact homogeneous manifolds; this analysis in terms of the Fourier analysis and the representation theory of such manifolds has been considered in Theorem \ref{compacthomogeneousmanifold}.   The equation \eqref{cardona-delcorral} provides a formula of the noncommutative residue for a class of global  pseudo-differential operators on compact Lie groups. In particular, our approach, provide formulae for the noncommutative residue of operators on the torus, in a different way to the   work of Pietsch \cite{Pietsch}.\\
\\
Now we present our result concerning to global operators on a manifold with boundary. In the formulation of our result we use the {\em global quantization} of pseudo-differential operators on compact manifolds with boundary due to J. Delgado, M. Ruzhansky and N. Tokmagambetov \cite{Ruz-Tok} which we briefly describe as follows. If $M$ denotes a compact manifold $M$ with boundary $\partial M$ and $L$ is a pseudo-differential operator on $M$  satisfying some boundary conditions on $\partial M$ and with a merely    discrete spectrum $\{\lambda_{\xi}\,|\, \xi\in \mathcal{I}\}$ (see Section \ref{boundary}), then every continuous linear operator $A$ acting on  a suitable domain $C^{\infty}_{L}(M),$ has associated a  function $\sigma_{A,L}:M\times  \mathcal{I}\rightarrow \mathbb{C}$ --called the {\em $L$-global symbol} of $A$-- satisfying
\begin{equation}
\sigma_{A,L}(x,\xi)=u_{\xi}(x)^{-1}A(u_{\xi})(x),\,\,\,\xi\in\mathcal{I},\,x\in M,
\end{equation}
where $u_{\xi}$ is the eigenvalue corresponding to $\lambda_{\xi}.$ In this case, the operator $A$ can be written in terms of the global symbol $\sigma_{A,L}(x,\xi)$ as
\begin{equation}
Af(x)=\sum_{\xi\in \mathcal{I}}u_{\xi}(x)\sigma_{A,L}(x,\xi)\widehat{f}(\xi),\,\text{ for }\,f\in  C^{\infty}_{L}(M),
\end{equation}
where $\widehat{f}:=\mathcal{F}_{L}(f)$ denotes the $L$-Fourier transform of $f$ introduced in \cite{Ruz-Tok}, (see also \eqref{Lfourier}).\\
An {\em $L$-Fourier multiplier} is a bounded linear operator $A:C^\infty_L(M)\to C^\infty_L(M)$ if it satisfies 
$$\mathcal{F}_L(Af)(\xi)=\sigma_{A,L}(\xi)\mathcal{F}_L(f)(\xi),\,\text{ for }\, f\in C^\infty_L(M),$$
for some function $\sigma_{A,L}:\mathcal{I}\rightarrow \mathbb{C}$ which depends only on the Fourier variable $\xi\in\mathcal{I}$, in this case $\sigma_{A,L}(\xi)$ corresponds with the $L$-symbol of $A$.\\
\\
With the above notation our result on the Dixmier traceability of operators on manifolds with boundary can be enunciated of the following way. We write $\mathcal{I}=\{\xi_{l}:l\in\mathbb{N}_{0}\}.$
\begin{theorem}
Let $M$ be a  compact manifold with boundary  $\partial  M.$ If $A:L^2(M)\to  L^2(M)$ is a bounded Fourier multiplier and $L$ is a self-adjoint operator on $L^{2}(M)$ as in Section \ref{boundary}, then we have the following assertions,
\begin{itemize}
\item $A$ is \textrm{Dixmier traceable} if and only if 
\begin{equation}\label{eq1mainboundary}
\tau'(A):=\lim_{N\rightarrow\infty}\frac{ 1  }{\log N}\sum_{l\leq N}  |\sigma_{A,L}(\xi_{l})|<\infty.
\end{equation}
In this case,    if $A$ is positive, $\tau'(A)=\textnormal{Tr}_{\omega}(A).$
\item Moreover, if $L$ is an operator of order $m$ satisfying the  Weyl Counting Eigenvalue Formula, that is, 
\begin{equation}\label{weyllawboundary}
N_L(\lambda):=\#\{l:|\lambda_{\xi_l}|^{\frac{1}{m}}\leq \lambda\}=C_{0}\lambda^{\dim M}+O(\lambda^{\dim M-1}),\,\,\lambda\rightarrow\infty,
\end{equation}
then $A$ is Dixmier traceable if only if (cf. \eqref{eq1mainboundary})
\begin{equation}
\tau'(A)=\frac{1}{\dim M}\lim_{N\rightarrow\infty}\frac{ 1  }{\log N}\sum_{ l: |\lambda_{\xi_{l}}|^{\frac{1}{m}}\leq  N}  |\sigma_{A,L}(\xi_{l})|<\infty.
\end{equation} 
In this case,  if $A$ is positive, $\tau'(A)=\textnormal{Tr}_{\omega}(A).$
\item  $A\in \mathcal{L}^{(p,\infty)}(L^2(M))$ if and only if
\begin{equation}
\gamma_{p}'(A):=\sup_{N\geq 1}\,\,N^{(\frac{1}{p}-1)} \sum_{l\leq N}|\sigma_{A,L}(\xi_l)|<\infty,
\end{equation}
for all $1<p<\infty.$ In this case, $\Vert A\Vert_{\mathcal{L}^{(p,\infty)}(L^{2}(M))}\asymp \gamma_{p}'(A).$ Moreover, if $L$  satisfies \eqref{weyllawboundary}, we have
\begin{equation}
\gamma_{p}'(A)\asymp\sup_{N\geq 1}\,\,N^{\dim M(\frac{1}{p}-1)} \sum_{l\,:\, |{\lambda_{\xi_l}}|^{\frac{1}{m}}\leq N}|\sigma_{A,L}(\xi_l)|.
\end{equation}
\end{itemize}
Moreover, if we assume the existence of a such $L$ as before, and $ \textit{A}=
 \begin{bmatrix}
   {}_{P_{+}} \\
   {}_{T}  \
 \end{bmatrix}
$ is an elliptic operator and positive in the Boutet de Monvel algebra of order $n$, then the operator $P_T$ (realization $P$ with respect to $T$) can be regarded as an operator in Ruzhansky- Tokmagambetov's calculus (see Remark \eqref{Eq:Realization-op}), and  any parametrix $R$ of $P_T$ is Dixmier's traceable with 

\begin{equation}\label{Eq:formula-bdy0}
\textnormal{Tr}_{\omega}(R)=\textnormal{res}\,(R)=\lim_{N\to \infty} \frac{1}{\ln N} \sum_{l\leq N}|(\sigma_{P,L}(\xi_l))^{-1}|,
\end{equation}
here $(\sigma_{P,L}(\xi_l))_{l\in\mathbb{N}_0}$ denotes the global symbol of the operator $P$ with respect to $L_{M}$. 
The expression is the same for all parametrices and independent of the choice of the boundary condition and independent on the average function $\omega$. 
\end{theorem}

As it was pointed out in \cite{weyllaw,Strichartz},  the condition \eqref{weyllawboundary} holds true for several classes of elliptic operators on manifolds with boundary and without boundary, which are not necessarily compact manifold. For our purposes,  \eqref{weyllawboundary} is not restrictive because such condition  implies that for some $s_{0}\in\mathbb{R},$
\begin{equation}
\tau(s_{0},L):=\sum_{\xi\in I}\langle \xi \rangle^{-s_0}<\infty,
\end{equation}
which is an important condition in the Ruzhansky-Tokmagambetov quantization (See Section \ref{boundary}) for manifolds with boundary.  For an exhaustive historical perspective on this subject we refer  to \cite{weyllaw}. Now, we present some references on the Dixmier traceability of pseudo-differential operators and related topics.  The problem of classifying   global pseudo-differential operators on certain ideals of bounded operators --as Schatten von Neumann clasess and Gr\"othendieck nuclear operators-- in terms of global symbols  has been considered  of a very satisfactory way in the references \cite{DR,DR1,DR3,DR4} and \cite{DR5}. Nevertheless, the problem of finding sufficient conditions on the symbol $\sigma$ of an pseudo-differential operator in order that the corresponding operator $A$ will be trace class, Hilbert Schmidt or Dixmier traceable is classical. If $H=L^{2}(\mathbb{R}^n)$ and we consider the pseudo-differential operator $A$ defined by the Weyl quantization,
\begin{equation}\label{weyl}
Af(x)=\int_{\mathbb{R}^n}\int_{\mathbb{R}^{n}}e^{ix\xi}\sigma_A(\frac{1}{2}(x+y),\xi)f(y)dyd\xi
\end{equation}
is well known that $\sigma_A(\cdot)\in L^{1}(\mathbb{R}^{2n}),$  implies that $A$ is class trace,  and $\sigma\in L^2(\mathbb{R}^{2n})$  gives $A$ Hilbert-Schmidt. In the framework of the Weyl-H\"ormander calculus of operators $A$ associated to symbols $\sigma$ in the $S(m,g)$-classes (see \cite{Hor2}), there exists two remarkable results. The first, due to H\"ormander, which asserts that if $\sigma\in S(m,g)$ and $\sigma\in L^{1}(\mathbb{R}^{2n})$  then $A$ is a trace class operator. The second, due to L. Rodino and F. Nicola expresses that if $\sigma \in S(m,g)$ and $m\in L^{1}_{w},$  (the weak-$L^1$ space), then $A$ is Dixmier traceable \cite{Rod-Nic}. Moreover, an open conjecture by Rodino and Nicola (see \cite{Rod-Nic}) says that $\sigma\in L^{1}_{w}(\mathbb{R}^{2n})$ gives an operator $A$ with finite Dixmier trace. More general regularized traces for pseudo-differential operators (see S. Scoot \cite{Scott} or S. Paycha \cite{Paycha} for complete a description) have been studied in different contexts, in  first instance, G. Grubb, \cite{GG05,GG05-2} has obtained results about the zeta regularized trace by using a resolvent approach; S. Paycha \cite{Paycha-Scott} has obtained important results about traces on meromorphic families of pseudo-differential operator;  B. Fedosov, F. Golse, E. Leichtnam, E.  Schrohe, \cite{Fedosov} have defined the noncommutative residue and its properties on  the algebra of pseudo-differential boundary value problems called Boutet de Monvel's algebra, in the case of manifolds with boundary;  G. Grubb and E. Schrohe in \cite{GSch04} have studied defect formulas for regularized traces on the Boutet de Monvel algebra, whose operators are choose to satisfy the so-called H\"ormander transmission property. Recently, one of the author has studied in \cite{CdC} the traciality property and uniqueness of the canonical trace for a class of operators not necessarily satisfying the transmission property.

Finally we explain the structure of our paper.
\begin{itemize}
\item In Section \ref{prel}, we present some preliminaries on the quantization of global pseudo-differential operators on compact Lie groups and general compact manifolds with or without boundary and we  briefly describe the Boutet de Monvel calculus. Also, we present some notions and known results about the Dixmier trace as well as the noncommutative residue.
\item In section \ref{proof} we prove our main results for operators acting on functions defined on compact manifolds without boundary.   The case of compact manifolds with boundary will be addressed in Section 
\ref{proof.boundary}.
\item In Section \ref{compacthm} we study the Dixmier traceability of multipliers  and we investigate the noncommutative residue in this setting.
\item We end our paper with some examples in Section \ref{examples}.
\end{itemize}

\section{Preliminaries}\label{prel}

In this section we present some basics facts about the Fourier analysis used trough this paper. Also, we recall some definition about the noncommutative residue (sometimes called Wodzicki's residue) and the Dixmier trace for pseudo-differential operators on compact manifolds with or without boundary. Trough of this article, $\mathcal{L}^{(1,\infty)}(H)$ denotes the Dixmier ideal of the set $\mathcal{L}(H)$ of bounded operators on an Hilbert space $H,$ $\textnormal{Spec}_p(A)$ denotes the pointwise spectrum of a (not necessarily bounded) linear operator $A$ on $H,$ and $\textnormal{Tr}_{\omega}(\cdot)$ is the Dixmier trace function on the ideal $\mathcal{L}^{(1,\infty)}(H).$

\subsection{Pseudo-differential operators on compact manifolds without boundary.}\label{Sec:PDOs-no-bdy}

We recall that for every $m\in \mathbb{R}$ and every open set $U\subset \mathbb{R}^n$, the H\"ormander class $S^{m}(U\times \mathbb{R}^n)$, (for a detailed description see \cite{Hor2}),  is defined  by functions satisfying the usual estimates
\begin{equation}\label{horm}
|\partial_{x}^{\alpha}\partial_{\xi}^{\beta}\sigma(x,\xi)|\leq C_{\alpha,\beta}\langle \xi\rangle^{m-|\beta|},
\end{equation}
for all $(x,\xi) \in T^*U \cong U\times \mathbb{R}^{n}$ and $\alpha,\beta\in \mathbb{N}^n$. Then, a pseudo-differential operator associated to a function $\sigma \in S^{m}(U\times \mathbb{R}^n)$ is a operator of the form
\begin{equation}
\sigma(x,D)f(x)=\int_{\mathbb{R}^n}e^{i x\xi}\sigma(x,\xi)\widehat{f}(\xi)d\xi,\,\,\text{ for }\,\,f\in C^{\infty}_{0}(U),
\end{equation}
where $C^\infty_0(U)$ denotes the set of smooth compact supported functions on $U$. The set of such operators is denoted by $\Psi^m(U\times\mathbb{R}^n)$. On a smooth manifold $M$ without boundary of dimension $n$, for every $m\in\mathbb{R},$ the H\"ormander class of order $m,$ $S^{m}(M)$ is defined by smooth functions on the cotangent $T^{*}M$ which in local coordinates coincides with symbols in some open sets $U$ of $\mathbb{R}^n$ satisfying inequalities as in  \eqref{horm}. We observe that the notion of symbol on arbitrary manifolds is of local nature. However, it is possible to define a notion of global symbol if we restrict our attention to the case of compact Lie groups. A such notion  was developed by M. Ruzhansky and V. Turunen in \cite{Ruz}. In this theory, every operator $A$ mapping $C^{\infty}(G)$ itself, where $G$ is a compact Lie group, can be described in terms of  representations of $G$ as follows.\\
\\
Let $\widehat{G}$ be the unitary dual of $G$ (i.e, the set of equivalence classes of continuous irreducible  unitary representations on $G$), the Ruzhansky-Turunen approach establishs that $A$ has associated a {\em matrix-valued global (or full) symbol} $\sigma_{A}(x,\xi)\in \mathbb{C}^{d_\xi \times d_\xi}$,  $[\xi]\in \widehat{G}$, on the noncommutative phase space $G\times\widehat{G}$ satisfying
\begin{equation}
\sigma_A(x,\xi)=\xi(x)^{*}(A\xi)(x).
\end{equation}
Then it can be shown that the operator $A$ can be expressed in terms of such a symbol as, \cite{Ruz},
\begin{equation}\label{mul}Af(x)=\sum_{[\xi]\in \widehat{G}}d_{\xi}\text{Tr}(\xi(x)\sigma_A(x,\xi)\widehat{f}(\xi)). 
\end{equation}  
An important feature in this setting is that the H\"ormander classes $\Psi^{m}(G),$ $m\in\mathbb{R}$  where characterized in \cite{Ruz,RWT} by the condition: $A\in \Psi^{m}(G)$ if only if its matrix-valued symbol $\sigma_{A}(x,\xi)$ as before satisfies 
\begin{equation}\label{equivalence}
\Vert \partial_{x}^{\alpha}\mathbb{D}^{\beta}\sigma_{A}(x,\xi)\Vert_{op} \leq C_{\alpha,\beta} \langle \xi\rangle^{m-|\beta|},
\end{equation}
for every $\alpha,\beta\in \mathbb{N}^n.$ For a rather comprehensive treatment of this quantization process we refer to \cite{Ruz}.\\
\\
The notion of global symbol on arbitrary compact manifolds is more delicate. This problem has been considered by J. Delgado and M. Ruzhansky in \cite{DR3,DR4}. Now, we explain this notion. If $M$ is a compact manifold without boundary with a volume element $dx$, and $E$ is a positive elliptic classical pseudo-differential operator of order $\nu>0,$ we say that $A:C^{\infty}(M)\rightarrow C^{\infty}(M)$ is $E$-invariant if $A$ and $E$ commutes. In this case we can associate a full symbol to $A$ in terms of the spectrum of $E$ as follows: the eigenvalues of $E$ form a positive sequence $\lambda_{j},$ which we label as 
\begin{equation}
0\leq \lambda_{0}\leq  \lambda_1\leq \lambda_2\leq\cdots \leq \lambda_n\leq \cdots.
\end{equation}
For every eigenvalue $\lambda_{j}$ the corresponding eigenspace $H_{j}=\textnormal{Ker}(E-\lambda_jI).$ If $\lambda_0=0$ then $d_{0}=\dim H_0=\dim \textnormal{Ker}(E).$ By the Spectral Theorem we can write
 $${L}^2(M)=\bigoplus_{j=0}^{\infty}H_{j}.$$
So, if we denote by $\{e_{jk} \}_{k=1}^{d_{j}}$ a basis of $H_j$ for $j\in\mathbb{N}$, where $d_j=\dim H_j,$ and for every $f\in L^2(M),$ we define the $E$-Fourier transform of $f$ (relative to the operator $E$ and the basis $\{e_{jk} \}_{k,j}^{d_{j}}$) at $j\in \mathbb{N}$ by
\begin{equation}
(\mathcal{F}_{E}f)(j):=\widehat{f}(j):=\begin{bmatrix} \langle f,e_{j1} \rangle_{L^{2}(M)}&\\ \langle f,e_{j2}\rangle_{L^{2}(M)}& \\ \vdots &\\ \langle f,e_{jd_j} \rangle_{L^{2}(M)}\end{bmatrix}_{d_j\times 1}.
\end{equation}
If $e_{j}$ denotes the column vector with entries $e_{j1},e_{j2},\cdots e_{jd_j}$ then, for every $f\in L^{2}(M),$ it follows from the Parseval Theorem  that $f$ has a Fourier series representation of the type, 
\begin{equation}
f(\cdot)=\sum_{l=0}^{\infty}\langle \widehat{f}(l),e_{l}(\cdot)\rangle_{d_l}.
\end{equation}
where $\langle \cdot,\cdot\rangle_{d_l}$ is the usual inner product on $\mathbb{C}^{d_l}.$ If $\{  \sigma(l)\}_{l\in\mathbb{N}_0}$ is a sequence of matrices such that for every $j,$ $\sigma(j)$ is a square matrix of order $d_j,$ the $E$-invariant operator $A$ associated  to $\{\sigma(l)\}_{l\in\mathbb{N}_0}$ is the operator on $C^{\infty}(M)$ defined by
\begin{equation}
(Af)(\cdot)=\sum_{l=0}^{\infty}\langle \sigma(l)\widehat{f}(l),e_{l}(\cdot)\rangle_{d_l}.
\end{equation}
Since $\widehat{Af}(l)=\sigma(l)\widehat{f}(l)$ we refer to $E$-invariant operators $A$ as {\em $E$-Fourier multipliers}, or simply by Fourier multipliers, multipliers or $E$-multipliers. An important feature in the Delgado-Ruzhansky's approach is that there exist formulae between the symbol $\sigma(l)$ of a $E$-multiplier $A$ from $\mathscr{D}'(M)$  into $\mathscr{D}'(M)$ (we use the notation $\mathscr{D}'(M)$ for the space of distributions in $M$) and its Schwartz kernel $K(x,y)$ (see \cite{Hor2}). In fact we have
\begin{equation}
\sigma(l)=\int_{M}\int_{M}K(x,y)Q_{l}(x,y)^{*}dxdy
\end{equation}
and
\begin{equation}
K(x,y)=\sum_{l=0}^\infty\textnormal{Tr}(\sigma(l)Q_{l}(x,y))
\end{equation}
where $Q_{l}(x,y)=\overline{e_l(y)}e_l(x)^{t}.$ We end this section with the following remark of global symbols on compact Lie groups which will be useful in our analysis related with the Dixmier trace for Fourier multipliers.
\begin{remark}\label{mainremarkglobal}
If $A$ is left-invariant (with respect to the Laplacian) then we have two notions of global symbols for $A.$ One is defined in terms of the representation theory of the group $G$ and we will denote this symbol by $(\sigma_{A}(\xi))_{[\xi]\in \widehat{G}},$ and the other one is that defined when we consider the compact Lie group as a manifold, and in this case the symbol will be denoted by $(\sigma_{A}(l))_{l\in \mathbb{N}_0}.$ The relation of this two symbols was discovered in \cite[pag. 25]{DR5}. Now, we describe this relation. In the setting of compact Lie groups the unitary
dual being discrete, we can enumerate the unitary dual as $[\xi_{j}],$ for $j\in\mathbb{N}_0.$  In this way we fix the orthonormal basis
\begin{equation}
\{e_{jk}\}_{k=1}^{d_j}=\{ d_{\xi_{j}}^{\frac{1}{2}}(\xi_{j})_{il}  \}_{i,l=1}^{d_{\xi_j}}
\end{equation}
where $d_{j}=d_{\xi_j}^2.$ Then, we have the subspaces $H_{j}=\textnormal{span}\{(\xi_j)_{i,l}:i,l=1,\cdots,d_{\xi_j}\}.$ With the notation above we have
$$\sigma_{A}(l)=
\begin{bmatrix}
    \sigma_{A}({\xi_{l}}) & 0_{d_{\xi_{l}}\times d_{\xi_{l}} } & 0_{d_{\xi_{l}}\times     d_{\xi_{l}} } & \dots  & 0_{d_{\xi_{l}}\times d_{\xi_{l}} } \\
     0_{d_{\xi_{l}}\times d_{\xi_{l}} } & \sigma_{A}({\xi_{l}})& 0_{d_{\xi_{l}}\times d_{\xi_{l}} } & \dots  & 0_{d_{\xi_{l}}\times d_{\xi_{l}} } \\
    \vdots & \vdots & \vdots & \ddots & \vdots \\
   0_{d_{\xi_{l}}\times d_{\xi_{l}} } & 0_{d_{\xi_{l}}\times d_{\xi_{l}} } &0_{d_{\xi_{l}}\times d_{\xi_{l}} } & \dots  & \sigma_{A}({\xi_{l}})
\end{bmatrix}_{d_{l}\times d_{l}}.$$
As a consequence of this discussion we obtain the following relation:
\begin{equation}\label{equivalenceofmatrices'}\textnormal{Tr}(|\sigma_{A}(l)|)=d_{\xi_l}\textnormal{Tr}(|\sigma_{A}(\xi_l)|),
\end{equation}
where $|A|:=\sqrt{A^\ast A}$ denotes the absolute value of the matrix $A$.
\end{remark}

\subsection{Global Fourier multipliers on compact homogeneous manifolds}
In order to present our theorem for multipliers on compact homogeneous spaces, we recall some definitions on the subject. Compact homogeneous manifolds can be obtained if we consider the quotient of a compact Lie groups $G$ with one of its closed subgroups  $K$\,\, --there exists an unique differential structure for the quotient $M:=G/K$--. Examples of compact homogeneous spaces are spheres $\mathbb{S}^n\cong \textnormal{SO}(n+1)/\textnormal{SO}(n),$ real projective spaces $\mathbb{RP}^n\cong \textnormal{SO}(n+1)/\textnormal{O}(n),$ complex projective spaces $\mathbb{CP}^n\cong\textnormal{SU}(n+1)/\textnormal{SU}(1)\times \textnormal{SU}(n)$ and more generally Grassmannians $\textnormal{Gr}(r,n)\cong\textnormal{O}(n)/\textnormal{O}(n-r)\times \textnormal{O}(r).$

Let us denote by $\widehat{G}_0$ the subset of $\widehat{G},$  representations of $G$, that are class I with respect to the subgroup $K$. This means that $\pi\in \widehat{G}_0$ if there exists at least one non trivial invariant vector $a$ with respect to $K,$ i.e., $\pi(h)a=a$ for every $h\in K.$ Let us denote by $B_{\pi}$ to the vector space of these invariant vectors and $k_{\pi}=\dim B_{\pi}.$ Now we follow the notion of Multipliers as in \cite{RR}. Let us consider the class of symbols $\Sigma(M),$ for $M=G/K,$ consisting of those matrix-valued functions \begin{equation}
\sigma:\bigcup\widehat{G}\rightarrow\bigcup_{n=1}^{\infty}\mathbb{C}^{n\times n}\,\,\text{ such that }\,\,\sigma(\pi)_{ij}=0\textnormal{ for all } i,j>k_{\pi}.
\end{equation}
Following \cite{RR}, a Fourier multiplier $A$ on $M$ is a bounded operator on $L^2(M)$ such that for some $\sigma_{A}\in \Sigma(M)$ satisfies
\begin{equation}\label{Eq.MultFourHomg}
Af(x)=\sum_{\pi\in\widehat{G}_0}d_{\pi}\textnormal{Tr}(\pi(x)\sigma_{A}(\pi)\widehat{f}(\pi)), \,\,\text{ for }\,f\in C^{\infty}(M),
\end{equation}
where $\widehat{f}$ denotes the Fourier transform of the lifting $\dot{f}\in C^{\infty}(G)$ of $f$ to $G,$ given by $\dot{f}(x):=f(xK),$ $x\in G.$
\begin{remark}
For every symbols of a Fourier multipliers $A$ on $M,$ only the upper-left block in $\sigma_{A}(\pi)$ of the size $k_\pi\times k_\pi$ can be not the  trivial matrix zero. 
\end{remark}

\subsection{Global pseudo-differential operators on compact manifolds with boundary}\label{boundary}

The natural analogue of the algebra of pseudo-differential operators on a closed manifold is the Boutet de Monvel's algebra (BdM's algebra), which will be described  in this section (we will use a kernel integral description of such algebra, but also we give a local symbolic description), where we follow the references \cite{BdM,BdM-66I,ES01}. It was defined in \cite{Fedosov}, by B. Fedosov, F. Golsche, E. Leichtmann and E. Schrohe  a linear functional trace, which coincides with the noncommutative residue whether the boundary is empty, and it is showed that such functional is the unique trace on BdM's algebra. Moreover, R. Nest and E. Schrohe showed in \cite{Nest-Schrohe} that such functional does not coincide with the Dixmier trace (instead of the closed case), however, they provide a particular set of operators in Boutet de Monvel's algebra where these two traces coincide. 
\\
\\
Let $M$ be a $n$-dimensional compact smooth manifold with boundary $\partial M$ which is embedded on a compact closed manifold $\Omega$, (we use $n$ to denote the dimension of the manifolds underlying instead of $\varkappa$, as in the case of a closed manifold). Boutet de Monvel's approach comprises operators acting on smooth functions over $M$ and over $\partial M$ as follow. An matrix operator of the form
\begin{equation}\label{BdM.opertor}
 \textit{A}=
 \begin{bmatrix}
   {}_{P_{+}+G} & {}_{K} \\
   {}_{T} & {}_{S} \
 \end{bmatrix}
 :C^\infty(\bar{M})\oplus C^\infty(\partial M)
 \rightarrow C^\infty(\bar{M})\oplus C^\infty(\partial M),
\end{equation}
is said that is a Boutet de Monvel's operator of order $m\in\mathbb{Z}$ and type $r\in\mathbb{N}$ if
\begin{itemize}
\item $P_+:=r^+Pe^+$ is an operator, called a {\em truncated operator}, built from a classical pseudo-differential operator $P$ on $\Omega$ (here $r^+$  corresponds to the restriction map from functions on $\Omega$ to functions on $M$  and the last one $e^+$ refers to the extension by zero from functions on $M$ to functions $\Omega$),  and the operator $P$ satisfies the {\em transmission property} (see \cite{BdM-66I}), namely, if $\sigma^P(x,\xi)$ denotes the local symbol satisfying \eqref{Eq:asym-homg}, (on local coordinates $(x,\xi)=(x',x_n,\xi',\xi_n)\in\mathbb{R}^{n-1}\times \mathbb{R}\times \mathbb{R}^{n-1}\times \mathbb{R}$ on the boundary) of $P$, then for all multi-indices $\alpha,\beta$ and any $j=0,1,\cdots,$ holds 
\begin{equation*}
 \partial^\alpha_{\xi'}\partial^\beta_{x_n} \sigma^P_{j}(x',0,0,+1)=
 (-1)^{j-|\alpha|} \partial^\alpha_{\xi'}\partial^\beta_{x_n}\sigma^P_j(x',0,0,-1).
\end{equation*}

\item $K$ is called a {\em Poisson Operator} of order $m$, and it can be defined by the integral operator
 $$(Kv)(x)=\int_{\mathbb{R}^{n-1}}e^{ix'\xi'}\tilde{k}(x',x_n,\xi')\mathcal{F}[v](\xi')d\xi',$$ with $\tilde{k}\in S^{m-1}(\mathbb{R}^{n-1},\mathcal{S}(\mathbb{R}_+)),$ in the sense of (\ref{Eq:symbol-estima-BdM}). In this case, the symbol associated with the operator $K$ is defined by
\begin{equation}
k(x',\xi',\xi_n):=\mathcal{F}_{x_n\to \xi_n} (\tilde{k}(x',x_n,\xi'))
\end{equation}
(here $\mathcal{F}_{x\to y}(f)$ means the Fourier transform of $f$ which transform the variable $x$ to $y$) and it is assumed an expansion in homogeneous functions $k_{m-1-j}(x',\xi',\xi_n)$ with respect to $\xi$ (for $|\xi|\geq 1$) of degree $m-1-j$.

\item $T$ is called a {\em Trace Operator} of order $m$, a type $r\in\mathbb{N}$, and it has the form $T=\sum\limits_{j=0}^{r-1} S_J\gamma_j+T',$ where $S_j$ is pseudo-differential operator on $\mathbb{R}^{n-1}$ of order $m-j$, and $T'$ can be defined by the integral operator
$$(T' u)(x')=\int_{\mathbb{R}^{n-1}} t'(x',x_n,\xi')\mathcal{F}_{x'\rightarrow\xi'}[u](\xi',x_n)dx_n d\xi',$$
with $\tilde{t}\in S^m(\mathbb{R}^{n-1}\times\mathbb{R}^{n-1},\mathcal{S}(\mathbb{R}_+)),$ in the sense of (\ref{Eq:symbol-estima-BdM}), similarly the symbol of $T$ is defined by 
\begin{equation}
t(x',\xi',\xi_n):=\mathcal{F}_{x_n\to \xi_n}( \tilde{t}(x',x_n,\xi')).
\end{equation}

\item $G$ is called a singular Green operator (s.G.o.) of order $m$ and type $r$; it is defined by an operator having the form $G=\sum_{j=0}^{r-1}K_j\gamma_j+G',$ where $K_j$ are Poisson operators; $\gamma_ju(x'):=D^j_{x_n}u(x' ,x_n)|_{x_n=0}$ (defines trace operators), and $G'$ can be defined by the integral operator 
$$ (G' u)(x):=\int_{\mathbb{R}^{n-1}}e^{ix'\xi'} \Big(\int_0^\infty \tilde{g}(x',x_n,y_n,\xi')\mathcal{F}_{x'\rightarrow \xi'} [u](\xi',y_n)\Big)d\xi',$$ 
with $\tilde{g}\in S^{m}(\mathbb{R}^{n-1}\times\mathbb{R}^{n-1};\mathcal{S}(\overline{\mathbb{R}}_+\times\overline{\mathbb{R}}_+))$, i.e. $\tilde{g}\in C^\infty(\mathbb{R}^{n-1}\times\mathbb{R}^{n-1};\mathcal{S}(\overline{\mathbb{R}}_+\times\overline{\mathbb{R}}_+))$ such that for all $\alpha,\beta,l,l',k,k'$
\begin{equation}\label{Eq:symbol-estima-BdM}
\|x^k_{n}D^{k'}_{x_n}y^{l}_{n}D^{l'}_{y_n}D^{\alpha}_{\xi'}D^{\beta}_{x'}\tilde{g}(x',x_n,y_n,\xi') \|_{L^2(\mathbb{R}^2_{++})}\leq C(x') \langle \xi' \rangle^{m+1-k+k'-l+l'-|\alpha|} 
\end{equation}
for a suitable positive constant $C(x')$ depending of $x'$. The function $\tilde{g}$ has an asymptotic expansion $\tilde{g}\sim \sum_{j\geq 0} \tilde{g}_{m-j}$ where $\tilde{g}_{m-j}$ satisfy the following homogenety property: for  $\lambda\geq 1$ and $|\xi'|\geq 1$, 
\begin{equation}\label{Eq:homo.Green.kernel}
\tilde{g}_{d-j}\Big(x',\frac{1}{\lambda}x_n,\frac{1}{\lambda}y_n,\xi'\Big)=\lambda^{m+2-j}\tilde{g}_{d-j}(x',x_n,y_n,\xi').
\end{equation}
In this case, the symbol associated with the operator $G$ is defined by 
\begin{equation}
g(x',\xi',\xi_n,\eta_n):=\mathcal{F}_{x_n\to \xi_n}\overline{\mathcal{F}}_{y_n\to \eta_n}\, (e^+_{x_n}e^+_{y_n}\tilde{g}'(x',x_n,y_n,\xi')),
\end{equation}
here $(\overline{\mathcal{F}}\phi)(\xi):=(\mathcal{F}\phi)(-\xi)$ for any $\phi\in L^2(\mathbb{R})$, the homogeneity property (\ref{Eq:homo.Green.kernel}) corresponds to 
$$g_{d-j}(x',\lambda \xi',\lambda \xi_n,\lambda \eta_n)=\lambda^{d-j}g(x',\xi',\xi_n,\eta_n),\,\, |\xi|\geq 1,\, \lambda\geq 1.$$
\item Finally, $S$ is a classical pseudo-differential operator on $\mathbb{R}^{n-1}$ of order $m$.
\end{itemize}
We have described Boutet de Monvel's algebra acting on functions with values in $\mathbb{R}$ or $\mathbb{C}$, which our setting of interest, however, such description can be done for a general setting as vector bundles.\\
\\
Now, let us mention an important result about the set of the boundary operators described above.

\begin{theorem}[{\em L. Boutet de Monvel, \cite{BdM}}]\label{Thm.BdM}
 Matrix operators as in (\ref{BdM.opertor}) form an algebra, called the {\em Boutet de Monvel  algebra}, more precisely, the composition $AB$ of an operator $A\in \mathcal{B}$ with all entries of order $m$ and type $r$ with an operator $B\in\mathcal{B}$ with all entries of order $m'$ and type $r'$ yields an operator in $\mathcal{B}$ of order $m +m'$ and type $\max\{m'+r,r'\}$.
\end{theorem}
Let us mention that Boutet de Monvel's algebra contains the parametrix operators of elements in such algebra  (whether this exists). A parametrix here means that $AB-I = S_1$ and $BA-I = S_2$ are regularizing operators; their types are $0$ and $m,$ respectively.\\ 
\\
On the one hand, in \cite{Fedosov} it is defined a linear functional on $\mathcal{B}$, which extends the noncommutative residue to the case of a manifold with boundary, 
\begin{eqnarray}\label{Eq:Noncomm.res.boundary}
\textnormal{res}\,(A)&=&\frac{(2\pi)^{-n}}{ n }\int_M\int_{|\xi|=1} p_{-n}(x,\xi)\, d\xi dx\nonumber\\ 
&&+\frac{(2\pi)^{-n+1}}{ n } \int_{\partial M}\int_{|\xi'|=1} \textnormal{tr}_n(g_{-n})(x',\xi')+s_{1-n}(x',\xi')\, d\xi' dx',\,\,
\end{eqnarray}
where $\textnormal{tr}_n(g_{-n})(x',\xi'):=\int_{0}^\infty \tilde{g}_{-n}(x',x_n,x_n,\xi')\, dx_n$. On the other hand, R. Nest and E. schrohe in \cite{Nest-Schrohe} study the noncommutative trace and the Dixmier trace for suitable sets of the Boutet de Monvel algebra. The authors obtain a first result about the Dixmier class. 

\begin{proposition}[{\em  \cite{Nest-Schrohe}, Proposition 2.1.}]
A bounded operator on the zero order Sobolev space $L^2(M)$ with range in the $n$-order Sobolev space $H^n(M)$ is an
element of $\mathcal{L}^{(1,\infty)}(M)$; moreover, if its range even is contained in $H^{n+1}(M)$ then it is trace class.
\end{proposition}

If $\mathcal{D}^{m}$ denotes the set of  matrix-operator $A$ in $\mathcal{B}$, as (\ref{BdM.opertor}),  for which  $P$ is an operator of order $m\in\mathbb{Z}$; $G$ is a. s.G.o. of order $m$ and type zero, $K$ is of order $m+1$ and type zero, and $S$ is an operator on the boundary of order $m+1$. In \cite{Nest-Schrohe} is showed that any operator $A$ in $\mathcal{D}^{-n}$ defines a bounded map $A:L^2(M)\oplus L^2(\partial M)\to H^{-n}(M)\oplus H^{1-n}(\partial M)$. Therefore, $\mathcal{D}^{-n}$ forms a subset of the Dixmier class $\mathcal{L}^{(1,\infty)}(H)$  where $H=L^2(M;E)\oplus L^2(\partial M;F)$. Moreover, the authors show in the following result a explicit formula for the Dixmier trace  in terms of the  terms appearing in the noncommutative residue (cf. (\ref{Eq:Noncomm.res.boundary})), for operators in $\mathcal{D}^{-n}$.

\begin{theorem}\label{Thm.N-Sch}[{\em \cite{Nest-Schrohe},Theorem 2.7.}]
For $A\in \mathcal{D}^{-n}$ acting on $H=L^2(M;E)\oplus L^2(\partial M;F)$, we have 
\begin{equation}\label{Eq.Dixmier-trace-Bdy}
\textnormal{Tr}_\omega (A)=\frac{(2\pi)^{-n}}{ n }\int_M \int_{|\xi|=1} p_{-n}(x,\xi)\, d\xi  dx+\frac{(2\pi)^{-n+1} }{n-1}\int_{\partial M} \int_{|\xi'|=1} s_{-n+1}(x',\xi')\, d\xi' dx'. 
\end{equation}
In particular, $\textnormal{Tr}_\omega (A)$ is independent of the averaging procedure $\omega$.
\end{theorem}

\begin{remark}
The Theorem \ref{Thm.N-Sch} shows that Conne's theorem does not hold for the Boutet de Monvel algebra, see (\ref{Eq:Noncomm.res.boundary}).
\end{remark}
Given $P:C^\infty(M)\to C^\infty(M)$ a differential operator of order $m>0$. In general one is interested in solve either the in-homogeneous problem
\begin{equation}\label{Eq.no-homg}
Pu=f\,\,\text{ on } M;\,\,\,\, Tu=g\,\,\text{ at }\,\, \partial M,
\end{equation}
for $f$ and $g$ given, or else the homogeneous problem
\begin{equation}\label{Eq.semi-homg}
Pu=f\,\,\text{ on } M;\,\,\,\, Tu=0\,\,\text{ at }\,\, \partial M.
\end{equation}
In order to treat problem (\ref{Eq.no-homg}), one consider operator of the form 
$$ \textit{A}=
 \begin{bmatrix}
   {}_{P}  \\
   {}_{T}  \
 \end{bmatrix}:H^m(M)\to L^2(M)\oplus H^{m-m'-1/2}(\partial M),$$ 
where $m'$ denotes de order of the trace operator $T$ and it is such that $m'<m$. For treat (\ref{Eq.semi-homg}) one studies the {\em realization} $P_T$ which is defined as the unbounded operator $P_T$ on $L^2(M)$, acting like $P$ over $$\mathcal{D}(P_T)=\{ u\in H^m(M)\, |\, Tu=0 \}.$$
\\
The ellipticity implies that there is a parametrix to $P_T$ in Boutet de Monvel’s calculus. It is of the form $B = (Q_++G\,\,\,\,K);$ the pseudo-differential part $Q$ is a parametrix to $P$, while $G$ is a singular Green
operator of order $−m$ and type zero and $K$ is a potential operator of order $−m$.

\begin{proposition}[{\em \cite{Nest-Schrohe}, Theorem 3.2.}]\label{Prop.parametrix}
Let $P_T$ be the above elliptic boundary value problem and
$B = (Q_++G\,\,\,\, K)$ its parametrix  . Then there is a regularizing singular Green operator $G_0$ of type zero, i.e. an integral operator with smooth kernel
on $M\times M$, such that $R=(Q_++G+G_0)$ has the following properties:
\begin{itemize}
\item[i.] $R$ maps $L^2(M)$ to $\mathcal{D}(P_T)$ and
\item[ii.] $RP_T-I$ and $P_TR-I$ are finite rank operators whose range consists of smooth functions.
\end{itemize}
\end{proposition}
It shown in Corollary 3.2. in \cite{Nest-Schrohe} that the parametrix $R$ in the above theorem is unique up to a regularizing singular Green operator of type $0$.


\begin{proposition}[{\em \cite{Nest-Schrohe}, Corollary 3.2.}]\label{Thm.Dix-res-NSCH}
If $m=n$ and 
$ \textit{A}=
 \begin{bmatrix}
   {}_{P_{+}}  \\
   {}_{T}  \
 \end{bmatrix}
$
is elliptic and positive, the Dixmier trace for an arbitrary parametrix $R$ to the operator $P_T$, we have
$$\textnormal{Tr}_\omega(R)=\frac{1}{(2\pi)^n n} \int_M\int_{|\xi|=1} (\sigma^P_{n}(x,\xi))^{-1}\,d\xi dx,$$
where $\sigma_n^P(x,\xi)$ denotes the homogeneous component for $\sigma^P(x,\xi)$ of degree $n$. The expression is the same for all parametrices and independent of the choice
of the boundary condition. Moreover, it coincides with the noncommutative residue $\textnormal{res}\,(R)$ for $R$.
\end{proposition}
By using non-harmonic analysis, M. Ruzhansky and N. Tokmagambetov in \cite{Ruz-Tok} give a different approach in terms of global symbols, which we summarize here. As above, $M$ denotes a compact manifold of dimension $n$ with boundary $\partial M$, and $L_M$ denotes the boundary value problem determined by a pseudo-differential operator $L$ of order $m$ on function in $L^2(M)$, on the interior of $M$, satisfying a suitable boundary conditions (BC) on $\partial M$. It is assumed the following conditions:
\begin{itemize}
\item The operator $L$ associated with the (BC) has a discrete spectrum $\{\lambda_\xi\in\mathbb{C}\,|\, \xi\in\mathcal{I}\}$, where  $\mathcal{I}$ is a countable subset of $\mathbb{Z}^k$ for some $k\geq 1$, and  the eigenvalues are ordered with the occurring multiplicities in the increasing order $|\lambda_j|\leq |\lambda_k|$ for $|j|\leq |k|$. The authors introduce the weight
$$\langle \xi \rangle:=(1+|\lambda_\xi|^2)^{1/2m},$$
and they assume that there exists a number $s_0\in\mathbb{R}$ such that 
\begin{equation}\label{s0}
\tau(s_{0},L):=\sum_{\xi\in\mathcal{I}} \langle \xi \rangle^{-s_0}<\infty.
\end{equation}
\item If $u_\xi$ denote the eigenfunction in $L^2(M)$ of $L$ associated with the (BC) corresponding to the eigenvalue $\lambda_\xi$ for $\xi\in\mathcal{I}$, so
$$Lu_\xi=\lambda_\xi u_\xi\,\,\text{ for }\,\, \xi\in\mathcal{I}.$$ The adjoint spectrum problem is
$$L^\ast  v_\xi=\overline{\lambda}_\xi v_\xi\,\,\text{ for }\,\, \xi\in\mathcal{I},$$
which is equipped with the conjugate boundary conditions (BC)$^\ast$. We can take biorthonormal system $\{ u_\xi\}$ and $\{v_\xi\}$ for $\xi\in\mathcal{I}$, i.e. $\|u_\xi\|=\|v_\xi\|=1$  for $\xi\in\mathcal{I}$ and  
\begin{equation*}
 \langle u_\xi, v_\eta\rangle=0\,\,\text{ for }\, \xi\neq \eta\,\,\text{ and } \langle u_\xi, v_\eta\rangle=1\,\,\text{ for }\, \xi=\eta,
\end{equation*}
where $\langle f,g \rangle:=\int_M fx)\overline{g(x)}\, dx$ is the usual product on the Hilbert space $L^2(M)$. In this context, we assume that the system $\{ u_\xi\,|\, \xi\in\mathcal{I}\}$ forms a basis of $L^2(M)$, i.e. for any $f\in L^2(M)$ there exists a unique series $\sum_{\xi\in \mathcal{I}} a_\xi u_\xi(x)$, so $\{ v_\xi\,|\, \xi\in\mathcal{I}\}$ is a basis of $L^2(M)$ too. 
\item Also we assume the functions $u_\xi$ and $v_\xi$ do not have zeros in $M$ for all $\xi\in\mathcal{I}$ and there exits  $C>0$ and $N\geq 0$ such that as $\langle  \xi \rangle\to \infty,$
\begin{equation*}
\inf_{x\in M} |u_\xi(x)|\geq C\langle \xi \rangle^{-N},\,\,\, \inf_{x\in M} |u_\xi(x)|\geq C\langle \xi \rangle^{-N}.
\end{equation*}
In this case the systems $\{ u_\xi\,|\, \xi\in\mathcal{I}\}$ is called a {\em WZ-system (without zeros system)}.
\end{itemize}
As it was mentioned in \cite{DRTk}, Remark 2.2., the condition given by WZ-system can be removed, however, under this situation the analysis underlying leads a matrix-valued version of the symbolic calculus, similar to Section \ref{Sec:PDOs-no-bdy}, which consists in vectors of eigenfunctions so that its elements do not all vanish at the same time- typical  examples of such situation is of operators on compact Lie groups, as it was described before.
\\
Now, we describe some elements involved in Ruzhansky-Tokmagambetov's calculus which will be needed in this paper.
The space 
\begin{equation}\label{Eq:R-T-dom}
C^\infty_L(M):=\bigcap_{k=1}^\infty \mathcal{D}(L^k)
\end{equation}
$\text{ where }\,\,\,\mathcal{D}(L^k):=\{ f\in L^2(M)\,|\, L^j f\in \mathcal{D}(L),\, j=0,1,\cdots ,k\},$
so that the boundary condition (BC) are satisfied by all the operators $L^j$. The Fr\'echet topology of $C^\infty_L(M)$ is given by the family of norms
$$\|f\|_{C^k_L}:=\max_{j\leq k} \| L^j f\|_{L^2(M)},\,\, k\in\mathbb{N}_0,\,\,f\in  C^\infty_L(M).$$
Similarly, it is defined $C^\infty_{L^\ast}(M)$ corresponding to the adjoint $L^\ast$ by 
$$C^\infty_{L^\ast}(M):=\bigcap_{k=1}^\infty \mathcal{D}((L^*)^k)$$ 
$\text{ where }\,\,\,\mathcal{D}((L^*)^k):=\{ f\in L^2(M)\,|\, (L^*)^j f\in \mathcal{D}(L),\, j=0,1,\cdots ,k\},$
which also has to satisfy the adjoint boundary conditions corresponding to the operator $L^\ast$. The Fr\'echet topology of $C^\infty_{L^\ast}(M)$ is given by the family of norms
$$\|f\|_{C^k_{L^\ast}}:=\max_{j\leq k} \| (L^\ast)^j f\|_{L^2(M)},\,\, k\in\mathbb{N}_0,\,\,f\in  C^\infty_L(M).$$
\\
Since $\{ u_\xi\}$ and $\{ v_\xi\}$ are dense in $L^2(M)$ that $C^\infty_{L}(M)$ and $C^\infty_{L^\ast}(M)$ are dense in $L^2(M)$.\\
\\
\textbf{$L$-Fourier transform:}
Let $\mathcal{S}(\mathcal{I})$ be the space of rapidly decreasing functions $\phi:\mathcal{I}\to \mathbb{C}$, i.e. for any $N<\infty,$ there exists a constant $C_{\phi,N}$ such that $|\phi(\xi)|\leq C_{\phi,N}\langle \xi\rangle^{-N}\,\,\text{for all } \xi\in\mathcal{I}.$ The space  $\mathcal{S}(\mathcal{I})$ forms a Fr\'echet space with the family of semi-norms $p_k(\phi):=\sup_{\xi\in\mathcal{I}} \langle\xi\rangle^k |\phi(\xi)|.$
The $L$-Fourier transform is a bijective homeomorphism $\mathcal{F}_L:C^\infty_L(M)\to  \mathcal{S}(\mathcal{I})$ defined by 
\begin{equation}\label{Lfourier}
(\mathcal{F}_L f)(\xi):=\hat{f}(\xi):=\int_M  f(x)\overline{v_\xi(x)}\, dx.
\end{equation}
The  inverse operator $\mathcal{F}^{-1}_L: \mathcal{S}(\mathcal{I})\to  C^\infty_L(M)$ is given by $(\mathcal{F}^{-1} h)(x):=\sum_{\xi\in \mathcal{I}} h(\xi) u_\xi(x)$, so that the Fourier inversion formula is given by
\begin{equation}
f(x)=\sum_{\xi\in \mathcal{I}} \hat{f}(\xi) u_\xi(x),\,\, f\in C^\infty_L(M).
\end{equation} 
Similarly, the $L^\ast$-Fourier transform is a bijective homeomorphism $\mathcal{F}_L:C^\infty_{L\ast}(M)\to  \mathcal{S}(\mathcal{I})$ defined by 
\begin{equation*}
(\mathcal{F}_{L^\ast} f)(\xi):=\hat{f}_\ast (\xi):=\int_M  f(x)\overline{u_\xi(x)}\, dx.
\end{equation*}
Its inverse $\mathcal{F}^{-1}_{L^\ast}: \mathcal{S}(\mathcal{I})\to  C^\infty_{L^\ast}(M)$ is given by $(\mathcal{F}^{-1}_{L^\ast}h)(x):=\sum_{\xi\in \mathcal{I}} h(\xi) v_\xi(x)$
so that the conjugate Fourier inversion formula is given by
\begin{equation}
f(x):=\sum_{\xi\in \mathcal{I}} \hat{f}_{\ast}(\xi) v_\xi(x),\,\, f\in C^\infty_{L^\ast}(M).
\end{equation}
The space $\mathcal{D}'_L(M):=\mathcal{L}(C^\infty_{L^*}(M,\mathbb{C}))$ of linear continuous functionals on $C^\infty_{L^*}(M)$ is called the space of {\em $L$-distribution}. By dualizing the inverse $L$-Fourier transform $\mathcal{F}^{-1}_L:\mathcal{S}(\mathcal{I})\to C^\infty_L(M),$ the $L$-Fourier transform extends uniquely to the mapping $$\mathcal{F}_L:\mathcal{D}'_L(M)\to\mathcal{S}'(\mathcal{I})$$
by the formula $\langle \mathcal{F}_L w,\phi\rangle:=\langle w, \overline{\mathcal{F}^{-1}_{L^*}\overline{\phi}}\rangle$ with $w\in \mathcal{D}'_L(M),$ $\phi\in\mathcal{S}(\mathcal{I})$.\\
\\
Moreover, the authors define Sobolev space adapted to $L_M$. The space $l^2_L:=\mathcal{F}_L\big( L^2(M)  \big)$ is defined as the image of $L^2(M)$ under the $L$-Fourier transform. Then the space of $l^2_L$ is a Hilbert space with the linear product
$$(a,b)_{l^2_L}:=\sum_{\xi\in\mathcal{I}} a(\xi)\overline{(\mathcal{F}_{L^\ast}\circ \mathcal{F}^{-1}b(\xi))}.$$
Then the space $l^2_L$ consists of the sequences of the Fourier coefficients of function in $L^2(M)$, in which Plancherel identity holds, for $a,b\in l^2_L$,
$$(a,b)_{l^2_L}=(\mathcal{F}^{-1}_L a,\mathcal{F}^{-1}_L b )_{L^2}.$$
For $f\in \mathcal{D}'_L(M)\cap \mathcal{D}'_{L^*}(M)$ and $s\in\mathbb{R}$, we say that 
$$f\in \mathcal{H}^s_L(M)\,\,\text{ if and only if }\,\, \langle \xi \rangle^s \hat{f}(\xi)\in l^2_L,$$
provided with the norm $$\|f \|_{\mathcal{H}^s_L}:=\big( \sum_{\xi\in\mathcal{I}}\langle \xi \rangle^{2s} \hat{f}(\xi)\overline{\hat{f}_*(\xi)} \big)^{1/2}.$$

\begin{proposition}[{\em \cite{Nest-Schrohe}, Proposition 6.3.}]\label{Prop:R-T-Sobolev}
For every $s\in\mathbb{R},$ the Sobolev space $\mathcal{H}^s_L(M)$ is a Hilbert space with the inner product 
$$(f,g)_{\mathcal{H}^s_L(M)}:=\sum_{\xi\in\mathcal{I}}\langle \xi \rangle^{2s}\hat{f}(\xi)\overline{\hat{g}_{*}(\xi)},$$
and the Sobolev space $\mathcal{H}^s_L(M)$ and $\mathcal{H}^s(M)$ are isometrically isomorphic.
\end{proposition}

\textbf{$L$-Quantization and $L$-symbol:} 
It can be shown that, see Theorem 9.2. in \cite{Ruz-Tok}, any continuous linear operator $A:C^\infty_L(M)\to C^\infty_L(M)$ can be expressed in terms the $L$-Fourier and its inverse transform as 
$$Af(x)=\sum_{\xi\in\mathcal{I}} u_\xi(x)\sigma_{A,L}(x,\xi) \hat{f}(\xi)$$
for every $f\in C^\infty_L(M)$ and $x\in M$, where  $\sigma_{A,L}(x,\xi)$ is called {\em the $L$-symbol} of $A$ and can be computed as (see \cite{Ruz-Tok}, Theorem 9.2.)
\begin{equation}\label{Eq:glob.sym.bdy}
\sigma_{A,L}(x,\xi)=u_\xi(x)^{-1}(A u_\xi)(x)\,\,\text{ for all }\, x\in  M\,\,\text{ and }\, \xi\in \mathcal{I}.
\end{equation}
In particular, if $A:C^\infty_{L}(M)\to C^\infty_{L}(M)$ be a continuous linear operator such that 
$$\mathcal{F}_{L}(Af)(\xi)=\sigma(\xi)\mathcal{F}_{L}(f)(\xi),\,\,\text{ for }\, f\in C^\infty_{L}(M)\,\,\text{ and }\, \xi\in\mathcal{I},$$
for some $\sigma:\mathcal{I}\to \mathbb{C}$; then $A$ is called a {\em $L$-Fourier multiplier}.\\
\\
The {\em adjoint operator} $A^*$ is defined by the equation
$$\langle A u_\xi, v_\xi\rangle_{L^2}=\langle u_\xi,A^* v_\xi\rangle_{L^2}.$$
It follows from the Parceval identity, see Proposition 6.1. in \cite{Ruz-Tok} that $A$ is an $L$-Fourier multiplier by $\sigma(\xi)$ if and only if $A^*$ is an $L^*$-Fourier multiplier by $\overline{\sigma(\xi)}.$\\
\\
Moreover, in \cite{Ruz-Tok}, the authors provided several results in the same asymptotic spirit of H\"ormander calculus for global symbols, e.g.  asymptotic formulas for the symbol of the adjoint operator $A^*$ and for the symbol of any parametrix $R$ of $A$,  in terms of the symbol of the operator $A$; for the symbol of the product $AB$ in terms of the symbol of $A$ and $B$.

\section{The Dixmier trace and the non commutative residue of Fourier multipliers on manifolds without boundary}\label{proof}
In this section we proof our main results concerning to the Dixmier trace and the noncommutative residue for  invariant operators on a closed manifold $M$ (a compact manifold without boundary). 
\begin{remark}
We are interested in bounded $E$-multipliers on $L^2(M).$ We recall that the following inequality
\begin{equation}
\sup_{l\in\mathbb{N}_0}\Vert \sigma_{A,E}(l)\Vert_{op}<\infty
\end{equation}
is a necessary and sufficient condition for the boundedness of $A$ on $L^2(M).$ This is an immediate consequence of the Plancherel Formula.  
\end{remark}
Now, we present our main theorem of this section.
\begin{theorem}\label{generalmanifold}
Let $M$ be a  compact manifold without boundary and let $E$ in the set $\Psi^{\nu}_{+e}(M)$ of positive elliptic pseudo-differential operators on $M.$ If  $A:L^2(M)\rightarrow L^2(M)$  is a bounded $E$-  invariant operator  with matrix-valued symbol $\sigma_{A,E}(l),$  then $A$ is \textrm{Dixmier traceable} if and only if 
\begin{equation}\label{eq1}
\tau(A):=\frac{1}{\dim(M)}\lim_{N\rightarrow\infty}\frac{ 1  }{\log N}\sum_{ l:(1+\lambda_{l})^{\frac{1}{\nu}}\leq N}\textnormal{Tr}(|\sigma_{A,E}(l)|)<\infty.
\end{equation}
Moreover,  if $A$ is positive, $\tau(A)=\textnormal{Tr}_{w}(A).$
\end{theorem}
\begin{proof}
We want to show that  \eqref{eq1} is a necessary and sufficient condition for the Dixmier traceability of a $E$-invariant operator $A$ with full symbol $\sigma_{A,E}(l).$ As a consequence of the relation $$\mathcal{F}_{E}(Af)(l)=\sigma_{A,E}(l)(\mathcal{F}_{E}f)(l),\,\,l\in \mathbb{N}_0,$$
we obtain
\begin{equation}\label{cup1}
\textnormal{Spec}_p(A)=\bigcup_{l\in\mathbb{N}_0}\textnormal{Spec}_p(\sigma_{A,E}(l)).
\end{equation}
In fact, if $\lambda\in \textnormal{Spec}_p(A)=\ker(A-\lambda I),$ there exists $u\neq 0,$ such that $Au=\lambda u.$ If we take the $E$-Fourier transform to both sides, we have $\sigma_{A,E}(l)\widehat{u}(l)=\lambda\widehat{u}(l).$ There exists $l_{0}\neq 0,$ such that $\widehat{u}(l_{0})\neq 0.$  Hence we obtain that $\lambda$ is a eigenvector of the matrix $\sigma(l_0).$ For the proof of the converse, let us assume that  for some $l_{0}\in \mathbb{N}_0,$ the complex number $\lambda\in \textnormal{Spec}_p(\sigma(l_0))$ is an eigenvalue with eigenvector $u(l_0).$ By contradiction let us assume that $\lambda$ is not a element of $\Sigma_p(A).$ It follows that $A$ is a Fourier multiplier it have not residual or continuous spectrum, therefore $\lambda$ is in the resolvent set of $A,$  in particular $A-\lambda I$ is  an injective operator. So we have that $(A-\lambda I)v=0$ implies that $v=0.$ Let us consider the sequence of matrices $(v_{l})_{l\in\mathbb{N}_0},$ defined by $v_{l}=0$ (zero matrix of dimension $d_{l}$) if $l\neq l_0,$ and $v_{l_0}=u(l_0).$ If $u$ is the inverse $E$-Fourier transform of this sequence, we have $\widehat{u}(l)=v_l.$ It is clear that for every $l,$ $\sigma_{A,E}(l)\widehat{u}(l)=\lambda\widehat{u}(l).$ By taking the inverse $E$-Fourier transform, we have $Au=\lambda u$ with $u\neq 0$ which is a contradiction.\\

Applying \eqref{cup1} to the operator $\sqrt[2]{A^* A}$ we obtain the following relation for the singular values of $A:$
\begin{equation}
s(A)=\bigcup_{l\in\mathbb{N}_0}s(\sigma_{A,E}(l)).
\end{equation}
Moreover, if $\lambda_{1,d_{l}},\lambda_{2,d_{l}},\cdots \lambda_{k_l,d_{l}}$ are the singular values of $\sigma_{A,E}(l)$ with corresponding multiplicities $m_{1,d_{l}},m_{2,d_{l}},\cdots m_{k_l,d_{l}},$
then 
\begin{equation}
m_{1,d_{l}}+m_{2,d_{l}}+\cdots m_{k_l,d_{l}}=d_{l}
\end{equation}
and 
\begin{equation}
m_{1,d_{l}}\lambda_{1,d_{l}}+m_{2,d_{l}}\lambda_{2,d_{l}}+\cdots m_{k_l,d_{l}}\lambda_{k_l,d_{l}}=\textnormal{Tr}[|\sigma_{A,E}(l)|],
\end{equation}
where $|\sigma_{A,E}(l)|=\sqrt[2]{\sigma_{A,E}(l)^* \sigma_{A,E}(l)}.$ With notations above, $A$ is Dixmier traceable (by definition) if and only if
\begin{equation}
\tau(A):=\lim_{N\rightarrow\infty}\frac{ \sum_{l\leq N}\textnormal{Tr}(|\sigma_{A,E}(l)|)  }{\log (\sum_{l\leq N}d_{l})}=\lim_{N\rightarrow\infty}\frac{ \sum_{l\leq N}\sum_{j=1}^{k_{l}}m_{j,d_{l}}\lambda_{j,d_{l}}  }{\log(\sum_{l\leq N}\sum_{j=1}^{k_{l}}m_{j,d_{l}})}<\infty.
\end{equation}
In this  case $\tau(A)=\textnormal{Tr}_w(A).$
Since
$$ \lim_{N\rightarrow\infty}\frac{ \sum_{l\leq N}\textnormal{Tr}(|\sigma_{A,E}(l)|)  }{\log(\sum_{l\leq N}d_{l})}=\lim_{N\rightarrow\infty}\frac{ \sum_{(1+\lambda_l)^{\frac{1}{\nu}}\leq N}\textnormal{Tr}(|\sigma_{A,E}(l)|)  }{\log(\sum_{(1+\lambda_l)^{\frac{1}{\nu}}\leq N}d_{l})},$$
by using the Weyl Eigenvalue Counting  formula for the operator $E$ we have
\begin{equation}
\sum_{(1+\lambda_l)^{\frac{1}{\nu}}\leq L}d_{l}=C_{0}L^{\varkappa}+O(L^{\varkappa-1}), 
\end{equation}
where $\varkappa=\dim(M).$ So, we obtain
\begin{equation}
\textnormal{Tr}_{w}(A)=\frac{1}{\dim(M)}\lim_{N\rightarrow\infty}\frac{ 1  }{\log N}\sum_{ (1+\lambda_{l})^{\frac{1}{\nu}}\leq N}\textnormal{Tr}(|\sigma_{A,E}(l)|)
\end{equation}
which is the desired result.
\end{proof}

An important fact in the formulation of our analysis is the following result proved by A. Connes in the 80's, (see \cite[pag. 307]{Connes}).

\begin{theorem}[A. Connes \cite{Connes}]\label{connestheorem}
Let $M$ be a closed manifold of dimension $\varkappa.$ Then every classical pseudo-differential operator $A$ of order $-\varkappa$ lies in $\mathcal{L}^{(1,\infty)}{(L^2(M))}$ and 
\begin{equation}
\textnormal{Tr}_{w}(A)=\textnormal{res}(A).
\end{equation}
\end{theorem}

Now, we apply the preceding theorem and some tools of representation theory in order to study the Dixmier trace and the noncommutative residue for operator on compact Lie groups.
 
 \begin{theorem}\label{compactliegroup}Let $M=G$ be a compact Lie group of dimension $\varkappa$ and $\widehat{G}$ be the unitary dual of $G.$ If we denote by $\sigma_{A}(x,\xi)\equiv \sigma_{A,-\mathcal{L}_G}(x,\xi)$ the matrix valued symbol associated to $A,$  then under the condition
\begin{equation}
\Vert \Delta_{\xi}^{\alpha}\partial_{x}^{\beta}\sigma_{A}(x,\xi) \Vert_{op}\leq C\langle \xi\rangle^{-\varkappa-|\alpha|},\,\,\,x\in G, [\xi]\in\widehat{G},
\end{equation}
the operator $A$ is Dixmier measurable. Moreover, if $A$ is left-invariant and positive, its Dixmier trace is given by
\begin{equation}\label{eq2}
\textnormal{Tr}_{w}(A)=\frac{1}{\dim(G)}\lim_{N\rightarrow\infty}\frac{ 1  }{\log N}\sum_{\xi: \langle \xi\rangle\leq N}d_{\xi}\textnormal{Tr}(|\sigma_{A}(\xi)|).
\end{equation} 
 \end{theorem}
\begin{proof}
We observe that by condition \eqref{equivalence}, the operator $A$ is a pseudo-differential operator of order $-\varkappa.$ It follows from  Theorem 5.2 in \cite{Ruz3} that $A$ is bounded on $L^{2}(G)$. Then, by   Connes' Theorem in the form of Theorem \ref{connestheorem}, $A$ is Dixmier traceable. Now, if $A$ is a Fourier multiplier on $G,$ from Remark \ref{mainremarkglobal} and the equation \eqref{equivalenceofmatrices'}, we deduce 
\begin{equation}\label{equivalenceofmatrices}\textnormal{Tr}(|\sigma_{A}(l)|)=d_{\xi_l}\textnormal{Tr}(|\sigma_{A}(\xi_l)|).
\end{equation}
Hence \eqref{eq1} becomes
\begin{align*}
\textnormal{Tr}_{w}(A) &=\lim_{N\rightarrow\infty}\frac{ \sum_{l\leq N}\textnormal{Tr}(|\sigma_{A,E}(l)|)  }{\log[\sum_{l\leq N}d_{l}]}\\
&=\lim_{N\rightarrow\infty}\frac{ \sum_{l\leq N}d_{\xi_l}\textnormal{Tr}(|\sigma_{A}(\xi_l)|)  }{\log[\sum_{l\leq N}d^2_{\xi_l}]}=\lim_{N\rightarrow\infty}\frac{ \sum_{\langle \xi \rangle\leq N}d_{\xi}\textnormal{Tr}(|\sigma_{A}(\xi)|)  }{\log[\sum_{\langle\xi\rangle\leq N}d^2_{\xi_l}]}
\end{align*}
which proves \eqref{eq2}. Finally, that \eqref{eq2} is a necessary and sufficient condition for the Dixmier traceability of $A$ it follows from \eqref{eq1}. By the Weyl Eigenvalue Counting  Formula (see \cite[pag. 539]{DR}), we have
\begin{equation}
\sum_{\langle \xi\rangle\leq N}d^2_{\xi}=C_{0}L^{\varkappa}+O(L^{\varkappa-1}),
\end{equation}
where $\varkappa=\dim(G).$ Hence
\begin{equation}
\textnormal{Tr}_{w}(A)=\frac{1}{\dim(G)}\lim_{N\rightarrow\infty}\frac{ 1  }{\log N}\sum_{\langle \xi\rangle\leq N}d_{\xi}\textnormal{Tr}(|\sigma_{A}(\xi)|)
\end{equation}
which completes the proof. 
\end{proof} 
With an analogous analysis as in the previous result we obtain the following characterization of invariant pseudo-differential operators belonging to the Marcinkiewicz ideal $\mathcal{L}^{(p,\infty)}(L^2(M)).$
\begin{theorem}\label{marcinkiewicz}
Let $M$ be a compact manifold without boundary and let $A$ be a bounded $E$-invariant operator on $L^2(M)$. Then  $A\in \mathcal{L}^{(p,\infty)}(L^2(M))$ if only if
\begin{equation}\label{teoremalpinfinity}
\gamma_{p}(A):=\sup_{N\geq 1}\,\,\,  N^{\dim M(\frac{1}{p}-1)}\sum_{l: (1+\lambda_{l})^{\frac{1}{\nu}}\leq N}\textnormal{Tr}(|\sigma_{A,E}(l)|)<\infty,
\end{equation}
for all $1<p<\infty.$ In this case $\gamma_{p}(A)\asymp\Vert A\Vert_{\mathcal{L}^{(p,\infty)}(L^2(M))}.$
\end{theorem}
\begin{proof}
With the notation above we can give a short proof for this fact. Indeed, for $N>0$ and $S_N:=\sum_{{l}\leq N}d_{l},$ $A\in  \mathcal{L}^{(p,\infty)}(L^2(M)) $ if and only if
\begin{equation}\label{pinfinity}
\sum_{l\leq N}\textnormal{Tr}(|\sigma_{A,E}(l)|)=O(S_N^{1-\frac{1}{p}}).
\end{equation}
By the Weyl counting eigenvalue function, \eqref{pinfinity} becomes 
\begin{equation}\label{pinfinity2}
\sum_{l:(1+\lambda_{l})^{\frac{1}{\nu}}\leq N}\textnormal{Tr}(\,|\sigma_{A,E}(l)|\,)=O((C_0  N^{\varkappa})^{1-\frac{1}{p}})=O(N^{\varkappa(1-\frac{1}{p})}),
\end{equation}
where $\varkappa=\dim(M).$ This observation completes the proof, indeed,
\begin{align*}
 \Vert A\Vert_{\mathcal{L}^{(p,\infty)}(H)}\asymp \sup_{N\geq 1}\,N^{\frac{1}{p}-1}\sum_{l:(1+\lambda_{l})^{\frac{1}{\nu}}\leq N}\textnormal{Tr}(|\sigma_{A,E}(l)|)=:\gamma_{p}(A)<\infty.
 \end{align*}
\end{proof} 
\begin{corollary} 
 If $M=G$ is a compact Lie group, and $A$ is a Fourier multiplier, \eqref{equivalenceofmatrices} gives that
 \begin{equation}
\sup_{N\geq 1}\,\,\, N^{(\frac{1}{p}-1)\dim G}\sum_{[\xi]: \langle \xi \rangle\leq N}d_{\xi}\textnormal{Tr}(|\sigma_{A}(\xi)|)<\infty,
 \end{equation}
 is a sufficient and necessary condition in order that $A\in \mathcal{L}^{(p,\infty)}(L^{2}(G))$ for all $1<p<\infty.$ 
\end{corollary}
 
\begin{remark}
 The classification of $E$-invariant operators $A$ on  Schatten-von Neumann classes $S_{p}(L^2(M))=\mathcal{L}^{(p,p)}(L^{2}(M))$ and on the class of $r$-nuclear operators $N_{r}(L^2(M))$ have been considered in \cite{DR,DR1,DR3,DR4,DR5}. A remarkable result in \cite{DR3} shows that the inequality $$ \sum_{[\xi]\in \hat{G}}\Vert a(\xi)\Vert_{S_{r}}^{r}d_{\xi}<\infty, $$
where $\Vert \sigma_A(\xi)\Vert_{S_{r}}:=\textnormal{Tr}(|\sigma_{A}(\xi)|^r)^{\frac{1}{r}},$ is a necessary and sufficient condition for the  $r$-nuclearity of $A$.  The $r$-nuclearity and boundedness   of global pseudo-differential operators on compact Lie groups in the general setting of Besov spaces has been studied by one of the authors in \cite{Cardona,CardonaBZ}.
 \end{remark}
Our result on the noncommutative residue can be enunciated in the following way: 
 
\begin{theorem}
Let $G$ be a compact Lie group of dimension $\varkappa$ and $A:C^{\infty}(G)\rightarrow D'(G)$ be a positive left invariant continuous linear  operator. Then
if $A\in \Psi^{-\varkappa}(G)$ is a positive classical pseudo-differential operator, then the noncommutative residue of $A,$ is given in terms of representations on $G$ by
\begin{equation}\label{resi1}
\textnormal{res}(A)=\frac{1}{\dim(G)}\lim_{N\rightarrow\infty}\frac{ 1  }{\log N}\sum_{\xi: \langle \xi\rangle\leq N}d_{\xi}\textnormal{Tr}(|\sigma_{A}(\xi)|)
\end{equation}
 \end{theorem} 
 \begin{proof}
From the Connes trace theorem in the form of Theorem \ref{connestheorem} we observe that the Dixmier trace of $A$ coincides with its noncommutative residue and \eqref{resi1} now follows from the equation \eqref{eq2}.
 \end{proof}
 
Now, we use our results on multipliers in order to provide formulae  for a class of non-invariant operators (operators with symbols depending on the spatial variable $x\in G$).
\begin{proposition}\label{noninvarianttrace}
Let $G$ be a compact Lie group and $\varkappa=\dim G.$ Let us assume that $A\in \Psi^{m}_{cl}(G)$ is a positive classic pseudo-differential operator with local symbol, admitting homogeneous components, of the form:
\begin{equation}
\sigma^A(x,\xi)\sim \sum_{j=0}^{\infty}a_{m-j}(x)\sigma^{A_{m-j}}(\xi).
\end{equation}
Then the non-commutative residue of $A$ is given by
\begin{equation}
\textnormal{res}(A)=\frac{1}{\dim(G)}\int_{G}a_{-\varkappa}(x)dx\times\lim_{N\rightarrow\infty}\frac{ 1  }{\log N}\sum_{\xi: \langle \xi\rangle\leq N}d_{\xi}\textnormal{Tr}(|\sigma_{A_{-\varkappa}}(\xi)|).
\end{equation}
\end{proposition}
\begin{proof}
By definition of the non-commutative residue, considering the measure of Haar of a compact Lie group is normalized, and the Connes theorem we have
\begin{align*}
\textnormal{res}(A)&:=\frac{1}{\varkappa (2\pi)^{\varkappa}}\int_{G}\int_{|\xi|=1}a_{-\varkappa}(x)\sigma^{A_{-\varkappa}}(\xi)d\xi\,dx\\
&=\int_{G}a_{-\varkappa}(x)\textnormal{vol}_\varkappa(x)\times \textnormal{res}(A_{-\varkappa})=\int_{G}a_{-\varkappa}(x)dx\times \textnormal{Tr}_\omega(A_{-\varkappa})\\
&=\int_{G}a_{-\varkappa}(x)dx\times\frac{1}{\dim(G)}\lim_{N\rightarrow\infty}\frac{ 1  }{\log N}\sum_{\xi: \langle \xi\rangle\leq N}d_{\xi}\textnormal{Tr}(|\sigma_{A_{-\varkappa}}(\xi)|).
\end{align*}
\end{proof}

We end this section with the following remark about the noncommutative residue for invariant operators.  In order to distinguish the local and global symbol, let us use an upper notation to denote the local symbol of $A$ by $\sigma^A(\xi)$, i.e, the symbol of $A$ defined as a section of the cotangent bundle $T^*G$ on $G$  (see, \cite{Hor2}), with asymptotic expansion $\sigma^A(\xi)\sim \sum \sigma^A_{m-j}(\xi)$ in positive-homogeneous components $\sigma^A_{m-j}$ of degree $m-j$ then  by combining of the equations \eqref{resi1} and \eqref{resi} we deduce the following relation between the global and the local symbol of $A.$

\begin{corollary} 
Let $A$ be a classical operator satisfying the conditions of the previous result. So,
\begin{equation}
\textup{res}\,(A):=\frac{1}{\varkappa(2\pi)^\varkappa}\int_{| \xi|=1}\sigma^A_{-\varkappa}(\xi)\, d\xi=\lim_{N\to\infty}\frac{1}{\varkappa\log N}\sum_{\xi: \langle \xi\rangle\leq N}d_{\xi}\textnormal{Tr}(|\sigma_{A}(\xi)|),
\end{equation}
where $\sigma^A_{-\varkappa}$ denotes the component of degree $-\varkappa$ in the asymptotic expansion of $\sigma^A.$ 
\end{corollary}

\section{The Dixmier trace of Fourier multipliers on compact homogeneous manifolds }\label{compacthm}

In Theorem \ref{generalmanifold}, on a closed manifold, we give a characterization of Dixmier traceable $E$-invariant operators through its global symbols. In this section we consider the special case of multipliers  on compact homogeneous manifolds and we describe such characterization in terms of the representation theory of such manifolds.  We generalize the corresponding assertion on multipliers given in Theorem \ref{compactliegroup}, but in contrast with such result we do not consider the case of H\"ormander classes as in such theorem. It is important to remark that in this section we  do not consider a relation with the noncommutative residue.\\

Now we present our main result of this section.
\begin{theorem}\label{compacthomogeneousmanifold}
Let us assume that $A$ is a Fourier multiplier as (\ref{Eq.MultFourHomg}) on $M:=G/K.$ Then $A$ is Dixmier traceable on $L^{2}(M)$ if and only if
\begin{equation}
\tau(A):=\frac{1}{\dim(M)}\lim_{N\rightarrow\infty}\frac{1}{\log N}\sum_{[\pi]\in \widehat{G}_0:\langle \pi\rangle\leq N}d_{\pi}\textnormal{Tr}(|\sigma_{A}(\pi)|)<\infty.
\end{equation}
In this case,  if $A$ is positive, $\tau(A)=\textnormal{Tr}_{w}(A).$
\end{theorem}
\begin{proof}
 If $A$ is a Fourier multiplier on $M,$ from Remark \ref{mainremarkglobal} and the equation \eqref{equivalenceofmatrices'}, we obtain 
\begin{equation}\label{equivalenceofmatrices}\textnormal{Tr}(|\sigma_{A,\mathcal{L}_{G/K}}(l)|)=d_{\pi_l}\textnormal{Tr}(|\sigma_{A}(\pi_l)|),
\end{equation}
where $\{[\pi_{l}]:l\in\mathbb{N}_0\}$ is a enumeration of $\widehat{G}_0.$
If we use \eqref{eq1} we obtain
\begin{align*}
\textnormal{Tr}_{w}(A) &=\lim_{N\rightarrow\infty}\frac{ \sum_{l\leq N}\textnormal{Tr}(|\sigma_{A,E}(l)|)  }{\log[\sum_{l\leq N}d_{l}]}\\
&=\lim_{N\rightarrow\infty}\frac{ \sum_{l\leq N}d_{\pi_l}\textnormal{Tr}(|\sigma_{A}(\pi_l)|)  }{\log[\sum_{l\leq N}d_{\pi_l}k_{\pi_l}]}=\lim_{N\rightarrow\infty}\frac{ \sum_{\langle \pi_l \rangle\leq N}d_{\pi}\textnormal{Tr}(|\sigma_{A}(\pi)|)  }{\log[\sum_{\langle\pi\rangle\leq N}d_{\pi_l}k_{\pi_l}]}
\end{align*}
which proves \eqref{eq2}. Finally, that \eqref{eq2} is a necessary and sufficient condition for the Dixmier traceability of $A$ it follows from \eqref{eq1}. By the Weyl Eigenvalue Counting  Formula (see \cite[pag. 7]{RR} or \cite{Shubin}), we have
\begin{equation}
\sum_{[\pi_{l}]\in \widehat{G}_0:\langle \pi_{l}\rangle\leq N}d_{\pi_l}k_{\pi_l}=O(C_{0}L^{\varkappa}),
\end{equation}
where $$\varkappa=\dim(M),\textnormal{   and   }C_{0}=\frac{1}{(2\pi)^\varkappa}\int_{\sigma_{1}(x,w)<1}dx\,dw,$$
and  the integral is taken with respect to the measure on the
cotangent bundle $T^*M$ induced by the canonical symplectic structure (c.f. \cite{Shubin}). Thus, we have,
\begin{equation}
\textnormal{Tr}_{w}(A)=\frac{1}{\dim(M)}\lim_{N\rightarrow\infty}\frac{1}{\log N}\sum_{[\pi]\in \widehat{G}_0:\langle \pi\rangle\leq N}d_{\pi}\textnormal{Tr}(|\sigma_{A}(\pi)|)
\end{equation}
which completes the proof. 
\end{proof}

\section{The Dixmier trace and the noncommutative residue of pseudo-differential operators on manifolds with boundary}\label{proof.boundary}

In this section we prove a result about the Dixmier trace of pseudo-differential boundary value problems where we link the two different approaches: representation by local and global symbols. \\
\\
Let us  note that  in the simplest case of a $L$-Fourier multiplier $A,$ the operator $A^\ast$ is a $L^\ast$-Fourier multiplier, but not necessarily a $L$-Fourier multiplier, and in general, the operators $A$ and $A^\ast$ satisfy different conditions on the boundary. Thus, unless the model operator $L_M$ (i.e. $L$ equipped with conditions on the boundary $\partial M$) will be considered self-adjoint, we can not compose $A$ and $A^\ast$ on their domains. Since we are interesting in the singular values of an operator $A,$ we will assume that $L$ is self-adjoint. In our proof, we assume the spectrum of $L$ enumerated by the set $\mathcal{I}=\{\xi_{l}:l\in\mathbb{N}_{0}\}.$

\begin{remark}
A Fourier multiplier on $C_{L}^{\infty}(M)$ is an operator with symbol $\sigma_{A,L}:\mathcal{I}\rightarrow \mathbb{C}.$ Trough of this section we consider bounded Fourier multipliers on $L^2(M).$ It is easy to see (by  the Plancherel formula which holds in a suitable spaces $l^2_{L}(\mathcal{I})$, see also \cite{Ruz-Tok}, Theorem 14.1.), that 
\begin{equation}
\sup_{l\in\mathbb{N}_0}|\sigma_{A,L}(\xi_{l})|<\infty
\end{equation}
is a necessary and sufficient condition for the boundedness of $A$ on $L^2(M).$
\end{remark}

\begin{theorem}\label{teoremboundary}
Let $M$ be a  compact manifold with boundary  $\partial  M.$ If $A:L^2(M)\to  L^2(M)$ is a bounded Fourier multiplier, and $L$ is self-adjoint operator on $L^{2}(M),$ then $A$ is \textrm{Dixmier traceable} if and only if 
\begin{equation}\label{eq2-boundary}
\tau(A)=\lim_{N\rightarrow\infty}\frac{ 1  }{\log N}\sum_{ l\leq N}  |\sigma_{A,L}(\xi_l)|<\infty.
\end{equation}
In this case,  if $A$ is positive, $\tau(A)=\textnormal{Tr}_{\omega}(A)$ and independent on the average function $\omega$. If in addition $L$  satisfies the Weyl Eigenvalue Counting Formula, we obtain
\begin{equation}
\textnormal{Tr}_{w}(A)=\frac{1}{\dim(M)}\lim_{N\rightarrow\infty}\frac{ 1  }{\log N}\sum_{ l\,:\, |\lambda_{\xi_l}|^{\frac{1}{m}}\leq N}|\sigma_{A,L}(\xi_l)|,
\end{equation}
where $m$ is the order of $L.$
\end{theorem}

\begin{proof}
We want to show that  \eqref{eq2} is a necessary and sufficient condition for the Dixmier traceability of a Fourier multiplier $A$ with full symbol $\sigma_{A,L}(\xi).$  It follows from (\ref{Eq:glob.sym.bdy}) that if $\sigma_{A,L}(x,\xi)=\sigma_{A,L}(\xi)$ does not depend on $x$, then 
\begin{equation}\label{Eq:eigen.val.bdy}
Au_\xi=\sigma_{A,L}(\xi)u_\xi,
\end{equation}
so that $\sigma_{A,L}(\xi)$ are the eigenvalues of the operator $A$ corresponding to the eigenfunctions $u_\xi$. Moreover, the pointwise spectrum of $A$ is (setting  $\mathcal{I}=\{\xi_{l}:l\in\mathbb{N}_{0}\}$)
\begin{equation}\label{cup1boundary}
\textup{Spec}_p(A)= \{ \sigma_{A,L}(\xi)\,:\, \xi \in\mathcal{I}\}=\{ \sigma_{A,L}(\xi_l)\,:\, l\in \mathbb{N}\}.
\end{equation}
In fact, if $\lambda\in \textup{Spec}_p(A)=\ker(A-\lambda I),$ with eigenfunction $u\neq 0$ associated with $\lambda$. If we take the $L$-Fourier transform $\mathcal{F}_{L}$ to both sides of $Au=\lambda u,$ we have $\sigma_{A,L}(\xi)\widehat{u}(\xi)=\lambda\widehat{u}(\xi).$ Since $\mathcal{F}_{L}$ is an isomorphism; there exists $l_{0}\in\mathbb{N},$ such that $\widehat{u}(\xi_{l_0})\neq 0.$  Therefore, $\lambda=\sigma_{A,L}(\xi_{l_0}).$ \\
\\
For the converse, let us consider the complex number $\lambda= \sigma_{A,L}(\xi_{l_0})$ for some $l_{0}\in \mathbb{N}.$   By contradiction let us assume that $\lambda$ is not a element of $\textup{Spec}_p(A).$ Since $A$ does not have residual or continuous spectrum, therefore $\lambda$ is in the resolvent set of $A,$  in particular $A-\lambda I$ is  an injective operator. So we have that $(A-\lambda I)v=0$ implies that $v=0.$ Let us consider the function in $v\in \mathcal{S}(\mathcal{I})$  defined by $v(\xi_{l})=0$  if $l\neq l_0,$ and $v(\xi_{l_0})=1.$ Because the $L$-Fourier transform is an isomorphism from $C^{\infty}_{L}(M)$ into $\mathcal{S}(\mathcal{I}),$ we can consider the inverse $L$-Fourier transform of $v,$   $u\in C^{\infty}_{L}(M) .$ We have $\widehat{u}(\xi_{l})=v(\xi_l).$ It is clear that for every $l,$ $\sigma_{A,L}(\xi_l)\widehat{u}(\xi_l)=\lambda\cdot\widehat{u}(\xi_l).$ By taking the inverse $L$-Fourier transform, we have $Au=\lambda u$ with $u\neq 0$ which is a contradiction.  \\
\\
It follows that $L$ is a self-adjoint operator and Theorem 10.10. in \cite{Ruz-Tok}, that $A^* A$ is well defined on $L^2(M)$ and a $L$-Fourier multiplier with symbol $\overline{\sigma_{A,L}(\xi)}\sigma_{A,L}(\xi)=|\sigma_{A,L}(\xi)|^2$. Applying \eqref{cup1boundary} to the operator $\sqrt[2]{A^* A}$ we obtain the following relation for the set $s(A)$ of singular values of $A:$
\begin{equation}
s(A)=\{|\sigma_{A,L}(\xi)|:\xi\in \mathcal{I}|\}.
\end{equation}
With notations above, we set $\mathcal{I}=\{\xi_{l}:l\in\mathbb{N}_{0}\}$ in such way  $s(A)$ forms an increasing sequences, $A$ is Dixmier traceable (by definition) if only if
\begin{equation}
\tau(A):=\lim_{N\rightarrow\infty}\frac{ \sum_{l\leq N}|\sigma_{A,L}(\xi_l)|  }{\log N}<\infty.
\end{equation}
Since
$$\tau(A)=\lim_{N\rightarrow\infty}\frac{ \sum_{{l}\leq N}|\sigma_{A,L}(\xi_l)|  }{\log N}=\lim_{N\rightarrow\infty}\frac{ \sum_{l:|\lambda_{\xi_l}|\leq N}|\sigma_{A,L}(\xi_l)|  }{\log\big(\sum_{l:|\lambda_{\xi_l}|\leq N}\,1\big)} ,$$
if  we assume that $L$ satisfies the  Weyl Eigenvalue Counting  Formula:
\begin{equation}
\sum_{\xi:|\lambda_{\xi}|\leq N}1=\#\{\xi_l:|\lambda_{\xi_l}|^{\frac{1}{m}}\leq N\}=C_{0}N^{\varkappa}+O(N^{\varkappa-1}) 
\end{equation}
where $\varkappa=\dim(M),$ we obtain
\begin{equation}
\textnormal{Tr}_{w}(A)=\frac{1}{\dim(M)}\lim_{N\rightarrow\infty}\frac{ 1  }{\log N}\sum_{l\,:\, |\lambda_{\xi_l}|\leq N}|\sigma_{A,L}(\xi_l)|,
\end{equation}
it follows that $s(A)$ is organized as an increasing sequence such that the result is independent of the average function $\omega$ chosen, which prove the result.
\end{proof}
We observe that an adaption of the proof of Theorem \ref{marcinkiewicz} gives the following result:
\begin{corollary}
 Let $A$ be a $L$-Fourier multiplier bounded on $L^2(M).$ Then $A\in \mathcal{L}^{(p,\infty)}(L^2(M))$ if and only if
\begin{equation}
\gamma_{p}'(A):=\sup_{N\geq 1}\,\,\,N^{(\frac{1}{p}-1)}\cdot\sum_{l\leq N}|\sigma_{A,L}(\xi_l)|<\infty,
\end{equation}
for all $1<p<\infty.$ In this case, $\Vert A\Vert_{\mathcal{L}^{(p,\infty)}(L^{2}(M))}\asymp \gamma_{p}'(A).$ Moreover, if $L$ satisfies \eqref{weyllawboundary}, we have
\begin{equation}
\gamma_{p}'(A)\asymp\sup_{N\geq 1}\,\,\,N^{\dim M(\frac{1}{p}-1)}\cdot\sum_{l\,:\, {|\lambda_{\xi_l}}|^{\frac{1}{m}}\leq N}|\sigma_{A,L}(\xi_l)|.
\end{equation}
\end{corollary}

In order to make a clear exposition of the our next result we need to establish a more clear relationship between Ruzhansky-Tokmagambetov approach and the local approach of boundary value problems.
\begin{remark}\label{Eq:Realization-op}
Any continuous operator $A:C^\infty(M)\to C^\infty(M)$ in Ruzhansky-Tokmagambetov's calculus ($C^\infty_L(M)$ is defined by means of $L_M$,  cf. (\ref{Eq:R-T-dom})) may be seem as an operator realization as follow:
We have
\begin{equation}
\mathcal{D}(L_M):=\{ u\in L^{2}(M)\, |\,\, u\,\,\text{ satisfying the  (BC)}\}.
\end{equation}
It follows from (\ref{Eq:R-T-dom}) that $A=A_L$ indeed is the realization of $A$ with respect to the space
$$\mathcal{D}(A):=C^\infty_L(M)=\{f\in L^2(M)\,|\, L^jf\in L^2(M),\,\, f\,\text{ satisfies (BC)},\,j\in \mathbb{N} \}.$$
\end{remark}

\begin{theorem}\label{Thm.Res-bdy}
Let  $ \textit{A}=
 \begin{bmatrix}
   {}_{P_{+}} \\
   {}_{T}  \
 \end{bmatrix}
$ be a positive elliptic operator in the Boutet de Monvel algebra of order $n$, where $P$ is a Fourier multiplier. We assume the existence a such operator $L$ satisfying Ruzhansky-Tokmagambetov's calculus in Section \ref{boundary}. Then the operator realization $P_T$ can be regarded as an operator in R-T's calculus (cf. Remark \ref{Eq:Realization-op}) and  parametrix $R$ of $P_T$  is Dixmier's traceable and 
\begin{equation}\label{Eq:formula-bdy}
\textnormal{Tr}_{\omega}(R)=\textnormal{res}\,(R)=\lim_{N\to \infty} \frac{1}{\ln N} \sum_{l\leq N}|(\sigma_{P,L}(\xi_l))^{-1}|,
\end{equation}
here $(\sigma_{P,L}(\xi_l))_{l\in\mathbb{N}_0}$ denotes the global symbol of the operator $P$ with respect to $L_{M}$, see Remark \ref{Eq:Realization-op}. 
The expression is the same for all parametrices and independent of the choice of the boundary condition and independent on the average function $\omega$. 
\end{theorem}

\begin{proof}
We first see that the operator realization $P_T$ belongs to Ruzhansky- Tokmagambetov's calculus, i.e. $P_T$ is the restriction of $P$ to the space $C^\infty_L(M)$ for some operator $L$ as in Section \ref{boundary}.  We know from Proposition \ref{Prop:R-T-Sobolev} that 
\begin{equation}\label{Eq:Sob_L}
H^s(M)=H^s_L(M),
\end{equation}
(i.e., the Sobolev space can be defined by the a operator $L$ and independent of the choice of such operator). Since $P$ is an operator satisfying the transmission property we have that $P$ preserves the space of smooth up to boundary functions $C^\infty(\overline{M})$. Therefore, the operator $P_T$ maps functions from $C^\infty_L(M)$ to $C^\infty_L(M)$, therefore $P_T$ restricted to $C^\infty_L(M)$  belongs to Ruzhansky- Tokmagambetov's calculus. We know that $C^\infty_L(M)$ and $C^\infty_{L^*}(M)$ are dense subspaces on $L^2(M)$. Moreover, from  (\ref{Eq:Sob_L}) we have that $P_{T}$ can be extended to (see also Remark \ref{Eq:Realization-op})
$$\mathcal{D}(P_T):=\{ u\in H^m(M)\, |\, Tu=0 \}.$$
Now, it follows from Proposition \ref{Prop.parametrix} that $R$ maps $L^2(M)$ to $\mathcal{D}(P_T)\subset L^2(M)$, thus $R$ lies in the Dixmier class. In this case, $R$ can be considered as Fourier Multiplier up to trace class operators. It follows from Theorem \ref{teoremboundary} that $R$ is Dixmier traceable if and only if the conditions given in (\ref{Eq:Dix.bdy}) holds, and in this case, $\textnormal{Tr}_\omega(R)=\tau(R)$. Now, the assertion in (\ref{Eq:formula-bdy}) follows from Proposition \ref{Thm.Dix-res-NSCH}. 
\end{proof}

Now, we provide examples of certain elliptic operators on the Boutet de Monvel algebra satisfying conditions of the Ruzhansky-Tokmagambetov calculus. For this we follows Ivrii \cite{Ivrii}.\\
Let $M$ be a compact manifold with smooth boundary $\partial M,$ $M$ of dimension $n\geq 2,$ and let $dx$ be a density on $M.$ Let $L$ be a second-order  elliptic pseudo-differential operator on $M,$ formally self-adjoint with respect to the inner product in $L^2(M,dx);$ let $X$ be a vector field transversal to $\partial M$ at every point, $\gamma$ the operator of restriction to $\partial M,$
\begin{equation}\label{eq:bdy.cond}
T=\gamma\,\,\,\text{  or }\,\,\, T=\gamma X+T_1 X
\end{equation}
a boundary operator where $T_1$ is a first-order classical pseudo-differential operator on $\partial M.$ If we suppose that the vector-operator $(L,T):C^\infty(M)\to C^\infty(M)\oplus C^\infty(\partial M)$ satisfies the Sapiro-Lopatinskii condition. Let us denote by $L_T:L^2(M)\to L^2(M)$ the clausure of $L_T$ on $L^2(M)$. Then $L_T$ is self-adjoint; its spectrum is discrete, with finite multiplicity and tends to $\pm \infty$, or $+\infty$, or $-\infty$.\\
Moreover, if $H$ is a closed subspace in $L^2(M)$ such that the self-adjoint projector $\Pi$ on $H$ belongs to Boutet de Monvel's algebra. It is assumed that $\Pi L_T\subset L_T\Pi$ (i.e., $H$ is an invariant subspace of $L_T$). Then the restriction $L_{T,H}:H\to H$ the $L_T$ to $H$ (here, $\mathcal{D}(L_{T,H})=\mathcal{D}(L_T)\cap H$) is self-adjoint; its spectrum is discrete, with finite multiplicity and tends to $\pm \infty$, or $+\infty$, or $-\infty$.\\
\\
In the following, we present two general frameworks which guarantees the exists of the auxiliary operator $L$ in the global calculus in  \cite{Ruz-Tok}, hence the Theorem \ref{Thm.Res-bdy} holds:
\begin{itemize}
\item
If $H=H^m(M)$ is the Sobolev space of positive order $m$, since the Sobolev Embedding Theorem follows that $H^m(M)$ is a closed subspace of $L^2(M)$. The operator $L_T$ defined by $$L_T:=I+\Delta_T:\mathcal{D}(L_T)\to L^2(M),$$ 
here $\mathcal{D}(L_T)=\{H^m(M)\,|\, Tu=0\}$ with $T$ as above, defines a second-order elliptic operator. Since Sobolev Embedding Theorem we have that $H^m(M)$ is an invariant subspace of $L_T$. Therefore, $L_T$ is a self-adjoint operator; its spectrum is discrete, with finite multiplicity and tends to $\pm \infty$, or $+\infty$, or $-\infty$.\\
Therefore, by the spectral theorem for unbounded operators there exists an orthonormal basis for $L^{2}(M)$ formed with eigenvalues of $L_T.$ We denote to the set of eigenvalues by $\{\lambda_{\xi_{l}}\},$ $l\in\mathbb{N}_{0}$ and the corresponding eigenvalues by $\{u_{\xi_l}\}_{l\in\mathbb{N}_0}.$ We can assume that $L_T$ satisfies the Weyl law (see \cite{Ivrii,weyllaw}):
\begin{equation}
N(\lambda)=\sum_{l:|\lambda_{\xi_l}|^{\frac{1}{m}}\leq \lambda}1=O(\lambda^{n}),\,\,\,\lambda\rightarrow\infty,
\end{equation}
then, if $l$ is such that $|\lambda_{\xi_l}|^{\frac{1}{m}}\leq \lambda$,
\begin{equation}
 \lambda^{-1}|\lambda_{\xi_l}|^{\frac{1}{2}} \leq \sum_{l:|\lambda_{\xi_l}|^{\frac{1}{2}}\leq \lambda} \lambda^{-1}|\lambda_{\xi_l}|^{\frac{1}{2}}  \leq \sum_{l:|\lambda_{\xi_l}|^{\frac{1}{2}}\leq \lambda}1=O(\lambda^{n}),\,\,\,\,\lambda\rightarrow\infty,
\end{equation}
and as consequence we have the estimate for  $\langle \xi_{l}\rangle=(1+|\lambda_{\xi_l}|)^{\frac{1}{2}},$  $l\rightarrow\infty,$ 
$$ 1<\langle\xi_{l} \rangle\asymp |\lambda_{\xi_l}|^{\frac{1}{2}}=O(\lambda^{n+1}),\,\,\,\,\,\,\,\langle\xi_{l} \rangle\rightarrow \infty. $$
Thus, it is clear that for some $-\infty<s_0<0,$  
\begin{equation}
\sum_{l\in\mathbb{N}_0} \langle \xi_{l}\rangle^{s_0} <\infty. 
\end{equation}
 Because the condition $WZ$ can be removed from the Ruzhansky - Tokmagambetov calculus (by merely technical reasons), we have that there exists a such operator $L.$

\item
Let $M$ be a manifold of dimension $2$. If $P:C^\infty(M)\to C^\infty(M)$ is a second order elliptic differential operator. We denote by $P_T$ the restriction of $P$ to the $\ker T$ with domain
\begin{equation}
D(P_{T})=\{ f\in C^{\infty}(M): Tf=0\}.
\end{equation}
Let us denote by $P_T$ the clausure of $P_T$ on $L^2(M).$ Hence, $P_T$ has discrete spectrum with finite multiplicity and tends to either to $\infty,$ or to $-\infty$ or to $\pm\infty.$ Setting $L=P$, by similar arguments as before, the operator realization $L_{M}=P_{T}$ satisfies the  condition in the Ruzhansky- Tokmagambetov calculus. In this case the boundary condition $(\textnormal{BC})$ defines the vectorial space given by
\begin{equation}
f\in L^2(M)\cap (\textnormal{BC}) \textnormal{ if and only if } f\in D(P_{T}).
\end{equation}
Because $C_{L}^{\infty}(M)\subset D(P_T)$ is a invariant subspace of $L$ we have that $L$ maps   $C_{L}^{\infty}(M)$ itself. Therefore, the second-order elliptic differential operator $L:= P$ (respectively $L_M:=P_T|_{C^\infty_{L}}$) belongs to the Ruzhansky-Tokmagambetov calculus and satisfies the conditions in Section \ref{boundary}.
\end{itemize}

\begin{corollary}\label{Cor:Bdy-exist L}
If $ \textit{A}=
 \begin{bmatrix}
   {}_{P_{+}} \\
   {}_{T}  \
 \end{bmatrix}$ is a positive elliptic operator in the Boutet de Monvel algebra of order $n$ with $T$ as (\ref{eq:bdy.cond}). In this case, we may take $L_T$ where $L=I+\Delta$. Then,  
any parametrix $R$ of $P_T$  is Dixmier's traceable and 
the formula given in (\ref{Eq:formula-bdy}) holds. In the particular case where $M$ is a manifold of dimension 2 and $P$ is a second-order differential operator as above, we have a formula independent of an auxiliary operator $L$,
\begin{equation}\label{Eq:Dix.bdy}
\textnormal{Tr}_{\omega}(R)=\textnormal{res}\,(R)=\lim_{N\to \infty} \frac{1}{\ln N} \sum_{l\leq N}|\sigma_{R}(\xi_l)|,
\end{equation} 
here $(\sigma_{R}(\xi_l))_{l\in\mathbb{N}_0}$ denotes the global symbol of the parametrix with respect to $L=P_T$.\\
\end{corollary}

\section{Examples}\label{examples}
In this section we provide some examples in relation with our main results. First, we consider a Dixmier traceable Bessel potential associated to $L.$ Also, we consider an example of a Dixmier traceable operator on the manifold with boundary $M=[0,1].$ Later, we consider the case of Bessel potential  on $\textnormal{SU}(2)\cong\mathbb{S}^3$ and $\textnormal{SU}(3)$ respectively. In our examples, we compute the Dixmier trace of certain operators showing closed formulae, but, we do not numerical approximations of these expressions in all cases.
\begin{example}
Let $M$ be a compact manifold with boundary $\partial M$ and $L$ as in the preceding sections. Let us consider the model operator $L_{M}$ self adjoint on $L^2(M).$ If we assume that $-L$ is a positive operator of order $\nu,$ and some positive real number $s_{1}$ satisfies 
\begin{equation}\label{s1}
\sum_{l\in\mathbb{N}_0}\langle \xi_{l}\rangle^{-s}=\infty,\,\,\,\sum_{l\in\mathbb{N}_0}\langle \xi_{l}\rangle^{-s'}<\infty
\end{equation}
for all $0<s\leq s_{1}<s'<\infty$ (this condition is natural in view of \eqref{s0}). Then with the notations above the operator $A:=(I-L)^{-\frac{s_1}{\nu}}$ is Dixmier traceable on $L^{2}(M).$ For the proof,  we combine that the global symbol of $A$ is given by $$\sigma_{A,L}(\xi_l)=(1-\lambda_{\xi_l})^{-\frac{s_1}{\nu}}\asymp\langle \xi_{l}\rangle^{-s_1}.$$ 
As consequence of \eqref{s1} we have that $\langle \xi_{l}\rangle^{-s_1}=O(l^{-1}),\,\,l\rightarrow \infty,$ and 
\begin{equation}
\sum_{l\leq N}\langle \xi_{l}\rangle^{-s_1}=O(\log N),\,\,N\rightarrow\infty.
\end{equation}
Hence, by Theorem \ref{teoremboundary} we obtain 
\begin{equation}
\textnormal{Tr}_{w}((I-L)^{-\frac{s_1}{\nu}})\asymp\lim_{N\rightarrow\infty}\frac{1}{\log N}\sum_{l\leq N}\langle \xi_{l}\rangle^{-s_1}<\infty. 
\end{equation}
\end{example}

\begin{example}
Consider $M=[0,1]$ and the operator $L=-i\frac{d}{dx}$ on $M^\circ=(0,1)$ with the domain 
$$D(L)=\{f\in W^1_2[0,1]\, |\, af(0)+bf(1)+\int_0^1 f(x)q(x)\, dx=0\},$$
where $a\neq 0$, $b\neq 0$, and $q\in C^1[0,1]$, here  $W^1_2[0,1]:=\{ f\in L^2[0,1]\, |\, f'\in L^2[0,1]\}$. Under the assumption that 
$$a+b+\int_0^1 q(x)\, dx=1$$
we have the inverse $L^{-1}$ exists and is bounded from $L^2[0,1]$ to $D(L)$. The operator $L$ has a discrete spectrum which can be enumerated by 
$$\lambda_j=2\pi  j-i\ln (-a/b)+\alpha_j,\,\text{ for }\,\,j\in\mathbb{Z}$$
where the sequence $\alpha_j$ satisfies that for any $\epsilon>0$, $\sum_{j\in \mathbb{Z}} |\alpha_j|^{1+\epsilon}<\infty$. If $m_j$ denotes the multiplicity of the eigenvalue $\lambda_j$, then $m_j=1$ for sufficiently large $|j|$, and the system of eigenfunctions are given by 
$$u_{jk}=\frac{(ix)^k}{k!}e^{i\lambda_jx},\,\text{ for }\,\, 0\leq k\leq m_j-1\,\, \text{ and }\,\,j\in\mathbb{Z}.$$ 
Let $j_0\in\mathbb{N}$ large enough so that $m_j=1$ for $|j|\geq j_0$. 
The global symbol $\sigma_{L,L}(x,j)$ of $L$ with respect to $L$ is given by 
$$\sigma_{L,L}(x,j) = \left\{
\begin{array}{c l}
\frac{x^{2k}}{k!^2}\lambda_j & x\in [0,1],\,\,0\leq k\leq m_j-1,\,\,|j|\leq j_o\\
\lambda_j  &\,\,|j|\geq j_0;
\end{array}
\right.
$$
for some $j_{0}\in \mathbb{N}.$ It follows from the functional calculus that $L^{-1}$ has a discrete spectrum given by $\lambda_j^{-1}$ with $j\in\mathbb{Z}$. Since $|\lambda_j|=O(j)$ for $j\in\mathbb{Z}$ we have $\sum_{|j|\leq N} |\lambda^{-1}_j|=O\big(\ln (2N+1)\big)$.  Therefore, is a $L^{-1}$ is a pseudo-differential operator which lies in $\mathcal{L}^{(1,\infty)}(L^2[0,1])$ with symbol $\sigma_{L^{-1},L}(\xi)=(\sigma_{L,L}(\xi))^{-1},$ $|\xi|\geq j_0,$ and 
\begin{align*}
\textnormal{Tr}_\omega(L^{-1})&=\lim_{N\to\infty} \frac{1}{\ln\, (2N+1)}\sum_{|j|\leq N} |\lambda_j^{-1}|\\
&= \lim_{N\to\infty} \frac{1}{\ln\, (2N+1)} \sum_{j_0\leq|j|\leq N} |(\sigma_{L,L}(x,j))^{-1}|\\
&=\lim_{N\to\infty} \frac{1}{\ln\, (2N+1)} \sum_{|j|\leq N} \frac{2\pi j+i\ln (-a/b)+\overline{\alpha_j} }{(2\pi\, j+ \textup{Re}\,(\alpha_j))^2+(\textup{Im}\, (\alpha_j)+\ln(b/a))^2}.
\end{align*}
\end{example}
\begin{example}
Let us assume that $A$ is a left-invariant operator on $\textnormal{SU}(2)\cong \mathbb{S}^3.$ We recall that the unitary dual of $\textnormal{SU}(2)$ (see \cite{Ruz}) can be identified as
\begin{equation}
\widehat{\textnormal{SU}}(2)\equiv \{ [t_{l}]:2l\in \mathbb{N}, d_{l}:=\dim t_{l}=(2l+1)\}.
\end{equation}
There are explicit formulae for $t_{l}$ as
functions of Euler angles in terms of the so-called Legendre-Jacobi polynomials, see \cite{Ruz}. In this case, if $A$ and  is Dixmier traceable with symbol  $\sigma ([t_l])\equiv \sigma(l) $ then,
\begin{align*}
\textnormal{Tr}_{w}(A) &=\lim_{N\rightarrow\infty}\frac{ 1}{\log \big( \sum_{l\leq \frac{N}{2},l\in \frac{1}{2}\mathbb{N}_0}d^2_l \big)}\, \sum_{l\leq \frac{N}{2},l\in \frac{1}{2}\mathbb{N}_0}d_{[t_l]}\textnormal{Tr}[|\sigma(l)|] \\
&=\lim_{N\rightarrow\infty}\frac{1}{\log \big( \sum_{l\leq \frac{N}{2},l\in \frac{1}{2}\mathbb{N}_0}(2l+1)^2 \big)}\,  \sum_{l\leq \frac{N}{2},l\in \frac{1}{2}\mathbb{N}_0}(2l+1)\textnormal{Tr}[|\sigma(l)|]\\
&=\lim_{N\rightarrow\infty}\frac{ 1   }{\log [\frac{1}{6}(N+1)(2N^2+7N+1)]}\sum_{l\leq \frac{N}{2},l\in \frac{1}{2}\mathbb{N}_0}(2l+1)\textnormal{Tr}[|\sigma(l)|].
\end{align*}
For example, if $A=(I-\mathcal{L}_{\textnormal{SU}(2)})^{-\frac{\varkappa}{2}},$ is the Bessel potential of order $-\varkappa=-\dim \textnormal{SU}(2)=-3,$ then $A$ is Dixmier traceable, $\sigma(l)=(1+l(l+1))^{-\frac{3}{2}}$ and 
\begin{align*}
&\textnormal{Tr}_{w}((I-\mathcal{L}_{\textnormal{SU}(2)})^{-\frac{3}{2}})\\
&=\lim_{N\rightarrow\infty}\frac{ 1   }{\log [\frac{1}{6}(N+1)(2N^2+7N+1)]}\sum_{{l\leq \frac{N}{2},l\in \frac{1}{2}\mathbb{N}_0}}(2l+1)^2[1+l(l+1)]^{-\frac{3}{2}}\\
&=\lim_{N\rightarrow\infty}\frac{ 1   }{\log [\frac{1}{6}(N+1)(2N^2+7N+1)]}\sum_{{n=0}}^{N}(n+1)^2\left[1+\frac{n}{2}(\frac{n}{2}+1)\right]^{-\frac{3}{2}}\\
&=\lim_{N\rightarrow\infty}\frac{ 1   }{\log [\frac{1}{6}(N+1)(2N^2+7N+1)]}\sum_{{n=1}}^{N+1}\frac{n^2}{8}[n^2+3]^{-\frac{3}{2}}\\
&\sim 0.03935... \textnormal{ (numerical evidence).}
\end{align*}
A similar result can be obtained if we consider the operator $(I-\mathcal{L}_{sub})^{-\frac{3}{2}}$ where $\mathcal{L}_{sub}:=D_{1}^2+D^{2}_{2}$ is  the sub-Laplacian on $\textnormal{SO}(3),$ here $D_1$ and $D_2$ are the derivatives in the corresponding variables to the Euler angles for $\textnormal{SO}(3)$  (see \cite{Ruz,Ruz3}).
\end{example}

\begin{example}
Let $A$ be a left-invariant pseudo-differential operator Dixmier traceable on $\textnormal{SU}(3).$ The Lie group $\textnormal{SU}(3)$ (see \cite{Fe1}) has dimension 8 and $3$ positive square roots $\alpha,\beta$ and $\rho$ with the property
\begin{equation}
\rho=\frac{1}{2}(\alpha+\beta+\rho). 
\end{equation}
We define the weights
\begin{equation}
\sigma=\frac{2}{2}\alpha+\frac{1}{3}\beta,\,\,\,\tau=\frac{1}{3}\alpha+\frac{2}{3}\beta.
\end{equation}
With the notations above the unitary dual of $\textnormal{SU}(3)$ can be identified with
\begin{equation}
\widehat{\textnormal{SU}}(3)\equiv\{\lambda:=\lambda(a,b)=a\sigma+b\tau:a,b\in\mathbb{N}_{0}, d_{\lambda}=\frac{1}{2}(a+1)(b+1)(a+b+2)\}.
\end{equation}
The Dixmier trace of $A,$ once denoted its full symbol by $\sigma(\lambda)\equiv\sigma(\lambda(a,b))$, is given by the expression
\begin{align*}
&\textnormal{Tr}_{w}(A) \\ &=\lim_{N\rightarrow\infty}\frac{  1 }{\log[\sum_{0 \leq a+b\leq N}d_{\lambda(a,b)}]}\times\sum_{0 \leq a+b\leq N} \frac{1}{2}(a+1)(b+1)(a+b+2)\textnormal{Tr}[|\sigma(\lambda(a,b))|]\\
&=\lim_{N\rightarrow\infty}\frac{  1 }{\log[\frac{1}{120}(N+1)(N+2)(N+3)(N+4)(2N+5)]} \times \\&\hspace{5cm} \sum_{(a,b):0 \leq a+b\leq N} \frac{1}{2}(a+1)(b+1)(a+b+2)\textnormal{Tr}[|\sigma(\lambda(a,b))|].
\end{align*}
Since, the eigenvalues of the Laplacian $\mathcal{L}_{\textnormal{SU}(3)}$ on $SU(3)$ are of the form
\begin{equation}
-c(\lambda)=-\frac{1}{9}(a^2+b^2+ab+3a+3b),
\end{equation}
the symbol of the operator $A=(I-\mathcal{L}_{\textnormal{SU}(3)})^{-\frac{8}{2}}$ (which is Dixmier traceable) is given by
\begin{equation}
\sigma(\lambda)=(1+c(\lambda))^{-4}.
\end{equation}
Hence,
\begin{align*}
&\textnormal{Tr}_{w}((I-\mathcal{L}_{\textnormal{SU}(3)})^{-\frac{8}{2}})\\
&=\lim_{N\rightarrow\infty}\frac{  1 }{\log[\frac{1}{120}(N+1)(N+2)(N+3)(N+4)(2N+5)]} \times \\&\hspace{1cm} \sum_{(a,b):0 \leq a+b\leq N} \frac{1}{36}(a+1)^{2}(b+1)^{2}(a+b+2)^{2}[1+(a^2+b^2+ab+3a+3b)]^{-4}.
\end{align*}
\end{example}
\begin{remark}
Similar examples to the considered above can be obtained if we take, for example, every $n$-torus $\mathbb{T}^n\cong \mathbb{R}^n/\mathbb{Z}^n$ or the compact Lie group $\textnormal{SO}(3)\cong\mathbb{R}\mathbb{P}^3,$ spheres $\mathbb{S}^n$ and arbitrary real and complex projective spaces.
\end{remark}

We end this paper with the following example on the noncommutative residue of a classical operator on Lie groups.
\begin{example}
Let us assume that $G$ is a compact Lie group of dimension $\varkappa,$ and let us define $A:=a(x)(-\mathcal{L}_{G})^{-\frac{\varkappa}{2}}$. We consider $a$ positive and smooth in order that $A$ will be a positive operator. The global symbol of $A$ is given by
\begin{equation}
\sigma_A(x,\xi):=a(x)\lambda_{[\xi]}^{-\frac{\varkappa}{2}}I_{d_\xi},
\end{equation}
and by using \eqref{cardona-delcorral} we have that
\begin{equation}
\textnormal{res}(A)=\frac{1}{\dim(G)}\int_{G}a(x)dx\times\lim_{N\rightarrow\infty}\frac{ 1  }{\log N}\sum_{\xi: \langle \xi\rangle\leq N}d_{\xi}^2\lambda_{[\xi]}^{-\frac{\varkappa}{2}}. 
\end{equation}
\end{example}

\noindent \textbf{Acknowledgments:} The authors would like to thank Alexander Cardona Gu\'io --from Universidad de los Andes-- and Carolina Neira --from Universidad Nacional de Colombia-- for several discussions about this topic.  This project was partially supported by Universidad de los Andes, Mathematics Department.

\bibliographystyle{amsplain}

\end{document}